\documentclass[lettersize,onecolumn]{IEEEtran}
\usepackage{amsmath,amsfonts}
\usepackage{algorithmic}
\usepackage{algorithm}
\usepackage{array}
\usepackage{textcomp}
\usepackage{stfloats}
\usepackage{url}
\usepackage{verbatim}
\usepackage{graphicx}
\usepackage[colorlinks=true,
            linkcolor=blue,
            citecolor=blue,
            urlcolor=blue]{hyperref}
\input{extrapackages}
\DeclareAcronym{MLE}{short = MLE, long = maximum likelihood estimation}
\DeclareAcronym{EM}{short = EM, long = expectation-maximization}
\DeclareAcronym{IMS}{short = IMS, long = imaging mass spectrometry}
\DeclareAcronym{HSI}{short = HSI, long = hyperspectral imaging}
\DeclareAcronym{GMM}{short = GMM, long = Gaussian mixture model}
\DeclareAcronym{SMM}{short= SMM, long = spiked mixture model}
\DeclareAcronym{RMT}{short = RMT, long = random matrix theory}
\DeclareAcronym{BBP}{short = BBP, long = Baik-Ben~Arous-Péché}
\DeclareAcronym{ESD}{short = ESD, long = empirical spectral distribution}
\DeclareAcronym{LSI}{short = LSI, long =Logarithmic Sobolev inequality}
\DeclareAcronym{MP}{short = MP, long =Mar\v{c}henko–Pastur}
\DeclareAcronym{SNR}{short = SNR, long =signal-to-noise ratio}
\DeclareAcronym{iid}{short = i.i.d., long =independent and identically distributed}
\DeclareMathOperator{\E}{\mathbb{E}}

\DeclareMathOperator{\Tr}{Tr}

\newcommand{\Var}{\operatorname{Var}}

\newcommand{\y}{\mathbf{y}}
\newcommand{\x}{\mathbf{x}}
\newcommand{\vb}{\mathbf{v}}
\newcommand{\zb}{\mathbf{z}}

\newcommand{\cb}{\mathbf{c}}
\newcommand{\db}{\mathbf{d}}
\newcommand{\Db}{\mathbf{D}}

\newcommand{\ub}{\mathbf{u}}
\newcommand{\wb}{\mathbf{w}}

\newcommand{\vnull}{\mathbf{0}} 

\newcommand{\e}{\boldsymbol{\varepsilon}}
\newcommand{\A}{\mathbf{A}}
\newcommand{\Z}{\mathbf{Z}}
\newcommand{\Sb}{\mathbf{S}}
\newcommand{\Pb}{\mathbf{P}}
\newcommand{\X}{\mathbf{X}}
\newcommand{\R}{\mathbf{R}}
\newcommand{\Bb}{\mathbf{B}}
\newcommand{\I}{\mathbf{I}}
\newcommand{\Hb}{\mathbf{H}}

\newcommand{\Sig}{\mathbf{\Sigma}}
\newcommand{\M}{\mathbf{M}}
\newcommand{\Lb}{\mathbf{L}}
\newcommand{\tb}{\mathbf{t}}
\newcommand{\g}{\mathbf{g}}

\usepackage{bbm}
\newcommand{\ind}{\mathbbm{1}}

\usepackage{amsthm}
\newtheorem{theorem}{Theorem}[section]
\newtheorem{lemma}[theorem]{Lemma}
\newtheorem{proposition}[theorem]{Proposition}

\DeclareMathOperator{\Lip}{Lip}

\usepackage{orcidlink}
\begin{document}

\title{Phase Transition of Eigenvalues of Covariances from the Spiked Mixture Model in High-dimensional Regimes}

\author{
    Paul-Louis Delacour\,\orcidlink{0009-0006-2191-6589},
    Raf Van de Plas\,\orcidlink{0000-0002-2232-7130}, \IEEEmembership{Member,~IEEE,}%
    
    \thanks{
        Paul-Louis Delacour is with the Delft Center for Systems and Control, Delft University of Technology, 2628 CD Delft, The Netherlands 
        (e-mail: p.l.delacour@tudelft.nl).
    }

    \thanks{
        Raf Van de Plas is with the Delft Center for Systems and Control, Delft University of Technology, 2628 CD Delft, The Netherlands,
        and with the Department of Biochemistry, Vanderbilt University, Nashville, TN 37232, U.S.A.,
        and with the Mass Spectrometry Research Center, Vanderbilt University School of Medicine, Nashville, TN 37232, U.S.A.
        (e-mail: raf.vandeplas@tudelft.nl).
    }
}

% \thanks{Manuscript received October 26, 2023; revised December 8, 2023.}}

\markboth{Preprint July~2026}%
{Shell \MakeLowercase{\textit{et al.}}: A Sample Article Using IEEEtran.cls for IEEE Journals}

%\IEEEpubid{0000--0000~\copyright~2023 IEEE}

\maketitle

\begin{abstract}
The \acf{SMM} has been introduced as a probabilistic model that generalizes the single-spike (Wishart) model to a mixture model form.
  With applications ranging from \acl{IMS} in the life sciences to hyperspectral imaging in computer vision, it is crucial to understand under which circumstances its signals can be recovered from noisy measurements.
  The highly multiplexed nature of these measurement types furthermore necessitates such analysis to hold in high-dimensional settings.
  In this paper, we prove that the extreme eigenvalues of the covariance matrix from the \ac{SMM} exhibit a phase transition in high-dimensional regimes.
  We show that this phase transition, and thus signal recovery by extreme eigenvalues, depends on several interacting factors: the correlation between spikes (\textit{i.e.}, how similar in content underlying signals are), the energy parameters (\textit{i.e.}, the absolute strength of each underlying signal), and the mixture probabilities (\textit{i.e.}, how likely it is to encounter each underlying signal).
  This work provides sharp information-theoretic bounds on the parameters needed to detect one or more spikes from extreme eigenvalues of the \ac{SMM} covariance matrix, and these guarantees could potentially impact any application of the \ac{SMM}.
  Understanding this interplay could serve as a tool for driving experimental design in analytical chemistry and life sciences. 
\end{abstract}


\section{Introduction}
Many high-dimensional systems and measurement types in, \textit{e.g.}, neural networks \cite{thamm2022random}, wireless communications systems \cite{wireless_Taherpour}, functional genomics \cite{nitzan2021revealing}, biomedical imaging \cite{VERAART2016394}, and finance \cite{laloux_finance_application} can be modeled as large random matrices whose spectral properties (eigenvalues and eigenvectors) encode essential information. 
As a result, tools from \ac{RMT} are increasingly used to analyze such systems and data. For example, recent work used these tools to analyze weight matrices or Jacobians in deep neural networks in order to understand generalization and structure \cite{benigni2021eigenvalue}.

As technological advances deliver increasingly high-dimensional measurement types, \textit{e.g.}, in fields such as spatial transcriptomics and \acf{IMS} \cite{Spraggins2026}, the need to exploit \ac{RMT} in a high-dimensional regime becomes essential. Therefore, in this work we focus specifically on random matrices whose dimensions grow without bound. A foundational result in this area is that for many classical ensembles (such as Wigner matrices), the empirical distribution of eigenvalues converges (in the weak sense) to a deterministic limit called the semicircle law \cite{erdHos2009semicircle}. Moreover, for sample covariance matrices, one similarly obtains in the limit the \ac{MP} law \cite{marcenko_pastur_1967}.

One particularly interesting phenomenon arises when a signal (\textit{i.e.}, a finite low‐rank deterministic component) is embedded in noise (\textit{i.e.}, a noisy high‐dimensional matrix). 
This scenario is relevant to questions of signal detectability, for example in the context of machine learning \cite{perry2016optimality}, and is known as the single-spike model.
Above certain thresholds, the signal or `spike' can be detected despite its noisy environment, through eigenvalues that escape from the bulk spectrum. This phenomenon is known as the \ac{BBP} phase transition \cite{bbp_phase_transition}.
Subsequent work has further extended this idea to finitely many spikes \cite{arous2024high}, community detection models \cite{two-community}, and beyond \cite{landau_singular_val,detection_spiked_rectangular}.

The \acf{SMM} \cite{SMM_IEEE} is a statistical model that generalizes the spiked Wishart model to a mixture model form. Equivalently, the covariance matrix from the \ac{SMM} model can be viewed as a sum of covariance matrices from single-spiked Wishart models of random size.
In this paper, we investigate the limiting spectral distribution of \ac{SMM} covariance matrices. 
Our analysis highlights how the correlation between spikes (\textit{i.e.}, how similar the underlying signals are), the energy parameters (\textit{i.e.}, the absolute strength of each underlying signal), and the mixture probabilities (\textit{i.e.}, how likely it is to encounter each underlying signal) govern the emergence of eigenvalues detaching from the bulk, a phenomenon often referred to as ``eigenvalue push-out''.
Such a result provides sharp information-theoretic bounds on the parameters needed to detect one or more spikes from the extreme eigenvalues of the covariance matrix. These guarantees potentially impact any application where the \ac{SMM} can be used, including ones in \ac{IMS} and \ac{HSI} \cite{SMM_IEEE}.

Section~\ref{sec:model+theorem+demo} introduces the \ac{SMM} and puts forward the \ac{SMM} phase transition theorem.
Furthermore, it includes practical demonstrations of the theorem, illustrating under which circumstances signals embedded in the noise are detectable from extreme eigenvalues.
In section~\ref{sec:cov-matrices+proof}, we provide a proof for the \ac{SMM} phase transition theorem, with supporting argumentation provided in supplementary sections~\ref{appendixA} and \ref{appendixB}.

\subsection{Related Works}

Asymptotic spectra of sample covariance matrices have been studied extensively. In the linear setting $\y=\Sig^{1/2}\zb$ with $\zb$ having \ac{iid} features, and $\Sig$ the associated covariance matrix, the limiting spectrum follows a (generalized) \acl{MP} law, and a \acl{BBP} phase transition governs the outliers: an isolated eigenvalue detaches from the bulk precisely when a population spike exceeds a critical threshold \cite{bbp_phase_transition, Bloemendal2013IsotropicLL, Knowles2017Anisotropic}. In \cite{wang2024nonlinear}, this picture was recently extended beyond settings where the dependence between features is encoded linearly through a covariance matrix $\Sig$. An arbitrary, possibly nonlinear dependence between features is allowed, requiring only that the quadratic forms $\y^T\A\y$ concentrate around their expectation uniformly over square matrices $\A$. 

In this work, we derive the phase transition specific to the \ac{SMM} using a low-rank perturbation approach akin to Benaych-Georges and Nadakuditi \cite{benaych2011eigenvalues}. Conditional on the selected subpopulation spike within the mixture model, the \ac{SMM}'s covariance matrix is a finite-rank perturbation of a white Wishart matrix. Together with the bi-orthogonal invariance of Gaussian noise, this facilitates that $K$ spikes are captured jointly by a single small matrix whose singularity locates all outliers at once, and which we reduce to the $K\times K$ matrix $\Lb$. This yields a couple of benefits over the deterministic-equivalent analysis of \cite{wang2024nonlinear}. First, the argument is lighter: rotational invariance reduces the problem to the $T$-transform of the limiting \ac{MP} bulk, which is explicit in our isotropic-noise setting and which we obtain by Gaussian concentration. This avoids the resolvent local laws and fluctuation-averaging machinery required by the coordinate-general analysis in \cite{wang2024nonlinear}. Second, treating the full rank-$K$ perturbation jointly lets us go beyond the distinct-spike regime of \cite{wang2024nonlinear}, whose results assume simple population spikes. Our approach covers repeated values of $\lambda_i(\Lb)$ by a perturbation argument and recovers each outlier at its correct multiplicity. The ability to handle repeating eigenvalues is important since such degeneracies are intrinsic to mixture models, where symmetric configurations of the subpopulations (\textit{e.g.}, balanced, equally spaced spikes) commonly produce repeated eigenvalues. Finally, our approach generalizes to regimes that do not satisfy the quadratic concentration assumption of \cite{wang2024nonlinear}. 
As a demonstration, in appendix~\ref{sec:rare_intense}, we show how our approach extends beyond scenarios with fixed signal strength spikes and can, in fact, handle a rare and intense regime in which the spike's signal strength diverges while the corresponding subpopulation probability vanishes, \textit{i.e.}, spikes of unbounded energy in a vanishing fraction of observations.

\subsection{Notation}
\begin{itemize}
    \item For $k \in \{1,\dots,n\}$, $\mathbf{e}_k \in \mathbb{R}^n$ denotes the $k$th standard basis vector.    
    \item For $\mathbf{S} \in \mathbb{M}_d(\mathbb{C})$ a self adjoint matrix, we denote its $d$ eigenvalues by
    \begin{align*}
        \lambda_1(\mathbf{S}) \geq \lambda_2(\mathbf{S}) \geq \cdots \geq \lambda_d(\mathbf{S}).
    \end{align*}
    \item When $d \to \infty$ and $i\geq 1$ is fixed, the extreme eigenvalues of a sequence $\mathbf{S}_d \in \mathbb{M}_d(\mathbb{C})$ refer to $\lambda_i(\mathbf{S}_d)$ (largest eigenvalues) or $\lambda_{d-i+1}(\mathbf{S}_d)$ (smallest eigenvalues).
    \item For a vector $\mathbf{v}$, $\|\mathbf{v}\|_2$ and $\|\mathbf{v}\|_\infty$ denote the Euclidean norm and the $\ell_\infty$ norm, respectively.
    \item For a matrix $\mathbf{A}$, $\|\mathbf{A}\|_{\mathrm{op}}$
    and $\|\mathbf{A}\|_{\mathrm{F}}$ denote the operator (spectral) norm
    and the Frobenius norm, respectively.
    \item For $u\in\mathbb{C}$, $E$ is a closed subset of $\mathbb{R}$, we denote the distance $d(u,E) = \min_{x\in E} |u-x|$.
\end{itemize}

\section{The \acf{SMM} and its phase transition}
\label{sec:model+theorem+demo}

The \ac{SMM}, introduced in \cite{SMM_IEEE}, is a generalization of the single-spike model to a mixture model form. We consider $n$ independent observations $\y_1, \ldots, \y_n \in \mathbb{R}^d$, each sampled from the \ac{SMM}:
\begin{align}
     & \y = 
    \begin{cases}
        \alpha \sqrt{\beta_1} \vb_1 + \e &\text{with probability } \pi_1\\
        \quad\vdots \\
        \alpha \sqrt{\beta_K} \vb_{K} + \e &\text{with probability } \pi_{K}\\
    \end{cases},\label{eq:model} \\
    &\alpha \sim \mathcal{N}(0,1),\: \e \sim \mathcal{N}(\vnull,\I),\: \sum_{k=1}^K \pi_k = 1,\nonumber\\
    &\beta_1, \ldots , \beta_K \in \mathbb{R}^{+}, \: \vb_1, \ldots, \vb_K \in \mathbb{R}^d, \|\vb_k\|_2=1 \nonumber,
\end{align}
where $\alpha$ is the random scaling factor of observation $\y$,
$\sqrt{\beta_k} \vb_k$ is the $k$-th subpopulation or spike with $\vb_k$ as the normalized spike signal and $\sqrt{\beta_k}$ reporting the strength of the signal,
$\e$ is the random noise of observation $\y$,
and $\pi_k$ is the probability of the $k$-th subpopulation.
Let $z\in\{1,\dots,K\} $ be a latent categorical variable indicating which of the spikes $\sqrt{\beta_1} \vb_1, \ldots , \sqrt{\beta_K} \vb_K$ was used to generate $\y$.

We will analyze the regime in which $d\to \infty, n \to \infty$, $d/n \to \gamma$, and $K$ remains fixed, hereafter referred to as the high-dimensional regime. 
In the following, we consider $\left\{\beta_k\right\}_{k=1}^K$ to be fixed, and further extend to diverging scenarios in appendix~\ref{sec:rare_intense}.
This paper adopts the Bayesian viewpoint in which the spikes are drawn from an arbitrary prior, but they are required to almost surely tend to a fixed correlation in the limit $d\to \infty$:
\begin{align}
    \vb_l \cdot \vb_k &\xrightarrow[d \to \infty]{\textit{a.s.}} \theta_{l,k} \in [0,1], \qquad\qquad\forall l,k \in [K]. \label{eq:asymmptotic_correlation}
\end{align}
A simple example of a distribution satisfying \eqref{eq:asymmptotic_correlation} in the case of $K=2$ is to sample $\vb_1, \vb_2$ as follows: 
\begin{equation}
    \left\{
    \begin{aligned}
        \vb_1 &= \ub_1/\|\ub_1\|_2 \\
        \vb_2 &= \frac{\theta \ub_1 + \sqrt{1-\theta^2}\,\ub_2}
                     {\|\theta \ub_1 + \sqrt{1-\theta^2}\,\ub_2\|_2}
    \end{aligned}
    \right.,
    \label{eq:vdefs}
\end{equation}
for $\ub_1, \ub_2 \sim \mathcal{N}(\mathbf{0},\I)$ and $\theta \in [0,1]$. The correlation then corresponds to 
\begin{align*}
    \vb_1 \cdot \vb_2 &= \frac{\theta \|\ub_1\|_2^2 + \sqrt{1-\theta^2}\: \ub_1 \cdot\ub_2 }{\|\ub_1\|_2\: \|\theta \ub_1 + \sqrt{1-\theta^2} \ub_2\|_2},
 \end{align*}
and, by using Gaussian concentration of measure, one gets $\vb_1 \cdot \vb_2 \xrightarrow[d \to \infty]{\textit{a.s.}} \theta$.
Furthermore, we consider the covariance matrices coming from \ac{SMM} model \eqref{eq:model}: 
\begin{align}
    \Sb_{d,n} = \frac{1}{n} \sum_{i=1}^n \y_i \y_i^\mathsf{T} \in \mathbb{R}^{d\times d},
    \label{eq:covariance_def}
\end{align}
along with their ordered eigenvalues 
\begin{align*}
    \lambda_1(\Sb_{d,n}) \geq \ldots \geq \lambda_d(\Sb_{d,n}) \geq 0.
\end{align*}
After introducing the \ac{SMM} model and establishing the high-dimensional regime setting, we posit a corresponding phase transition theorem. 

\begin{theorem}[\ac{SMM} phase transition]
\label{thm:SMM_phase_transition}
The extreme eigenvalues of $\Sb_{d,n}$ \eqref{eq:covariance_def} exhibit the following behavior as $n\to \infty$ and $d\to \infty$ with $d/n \to \gamma$. 
Consider $\mathbf{L} \in \mathbb{R}^{K\times K}$ the Gram limiting matrix with entries $\mathbf{L}_{l,m} = \theta_{l,m} \sqrt{\pi_l \pi_m \beta_l \beta_m}$. 
\begin{itemize}
    \item For $1\leq i \leq K$,
    \begin{align*}
    \lambda_i(\Sb_{d,n}) \xrightarrow{\textit{a.s.}}
    \begin{cases}
        T^{-1}(1/\lambda_i(\Lb)) &\quad\text{if } \lambda_i(\Lb) > \sqrt{\gamma} \\
        (1 + \sqrt{\gamma})^{2} &\quad\text{otherwise},
    \end{cases}; \text{ and}
    \end{align*}    
    \item for $i > K$, 
    \begin{align*}
        \lambda_i(\Sb_{d,n}) &\xrightarrow{\textit{a.s.}} (1+\sqrt{\gamma})^2.
    \end{align*}
\end{itemize}
Here,
\begin{align*}
    T(z) &= \int \frac{t}{z-t} d \mu_{\mathrm{MP}(\gamma)}(t) &\text{for } z \in \mathbb{C} \setminus [a,b],
\end{align*}
is the $T$-transform of $\mu_{\mathrm{MP}}(\gamma)$, and $\mu_{\mathrm{MP}(\gamma)}$ is the \acl{MP} distribution with parameter $\gamma > 0$. Its density is 
\begin{align*}
    d \mu_{\mathrm{MP}(\gamma)}(x) = \frac{1}{2\pi \gamma x} \sqrt{(b-x)(x-a)}\ind_{[a,b]}(x) d x + \max\left(0,1-\frac{1}{\gamma}\right)\delta_0,
\end{align*}
where $\ind_{[a,b]}$ is the indicator function, $a := (1-\sqrt{\gamma})^2$ and $b :=(1+\sqrt{\gamma})^2$ denote the edges of the bulk of the distribution, and $\delta_0$ is the Dirac delta at location  $0$.
The inverse $T$-transform associated with $\mu_{\mathrm{MP}(\gamma)}$ is
\begin{align*}
    T^{-1}(l) = (1+1/l)(1+\gamma l).
\end{align*}
Note that for $l\in \mathbb{R}^+$ we indeed have $T^{-1}(1/l) \geq b$ with equality only for $l=\sqrt{\gamma}$.
\end{theorem}

Importantly, the single-spike case of Theorem~\ref{thm:SMM_phase_transition} (equivalent to $K=1$) is known as the \ac{BBP} phase transition \cite{bbp_phase_transition} and is well-studied. Here, we extend the characterization of the phase transition to a mixture model with $K\geq1$ spikes, in which the underlying spikes almost surely tend to some correlation.

Also note that the Gram limiting matrix $\Lb$ in Theorem~\ref{thm:SMM_phase_transition} is positive semi-definite given that the matrix of correlations $\{\theta_{l,m}\}_{1\leq l,m\leq m}$ is the limit of Gram matrices.

Before addressing the proof of Theorem~\ref{thm:SMM_phase_transition}, we illustrate this result empirically. We conducted a set of synthetic experiments in which datasets were constructed from the \ac{SMM} model in \eqref{eq:model} with $n=2000$, $d=1000$, $K=2$, and $\gamma=0.5$, and for which we sampled the $\vb_1,\vb_2$ spikes according to \eqref{eq:vdefs} and computed the covariance matrix~\eqref{eq:covariance_def}. 
In Figure~\ref{fig:eig-smm}, we plot the eigenvalue distribution of the sample covariance matrix of a particular dataset. In this example dataset, the signal strengths of the underlying spikes were set to $\beta_1 = 4$ and $\beta_2 = 5$, the probabilities of encountering these spikes were not equal, namely $\pi_1 = 0.4$ and $\pi_2 = 0.6$, and the correlation between the spikes' signal content was $\theta = 0.4$. For these particular values of the hyperparameters, we observe two eigenvalues successfuly popping out of the bulk of the \acl{MP} distribution. In Figure~\ref {fig:smm-one-param-change}, we demonstrate that not all hyperparameter value combinations lead to two escaping eigenvalues, but only a subset of them. 
This highlights the importance of understanding the interplay between the different \ac{SMM} hyperparameters to determine under which circumstances spectral methods can detect the presence of underlying signals, and under which circumstances signals become effectively unrecoverable by spectral methods. 
This information will be crucial for effective experimental design in noisy environments and applications.

\begin{figure}
    \centering
    \begin{subfigure}{\textwidth}
        \centering
        \captionsetup{width=\textwidth}
        \includegraphics[width=0.6\textwidth]{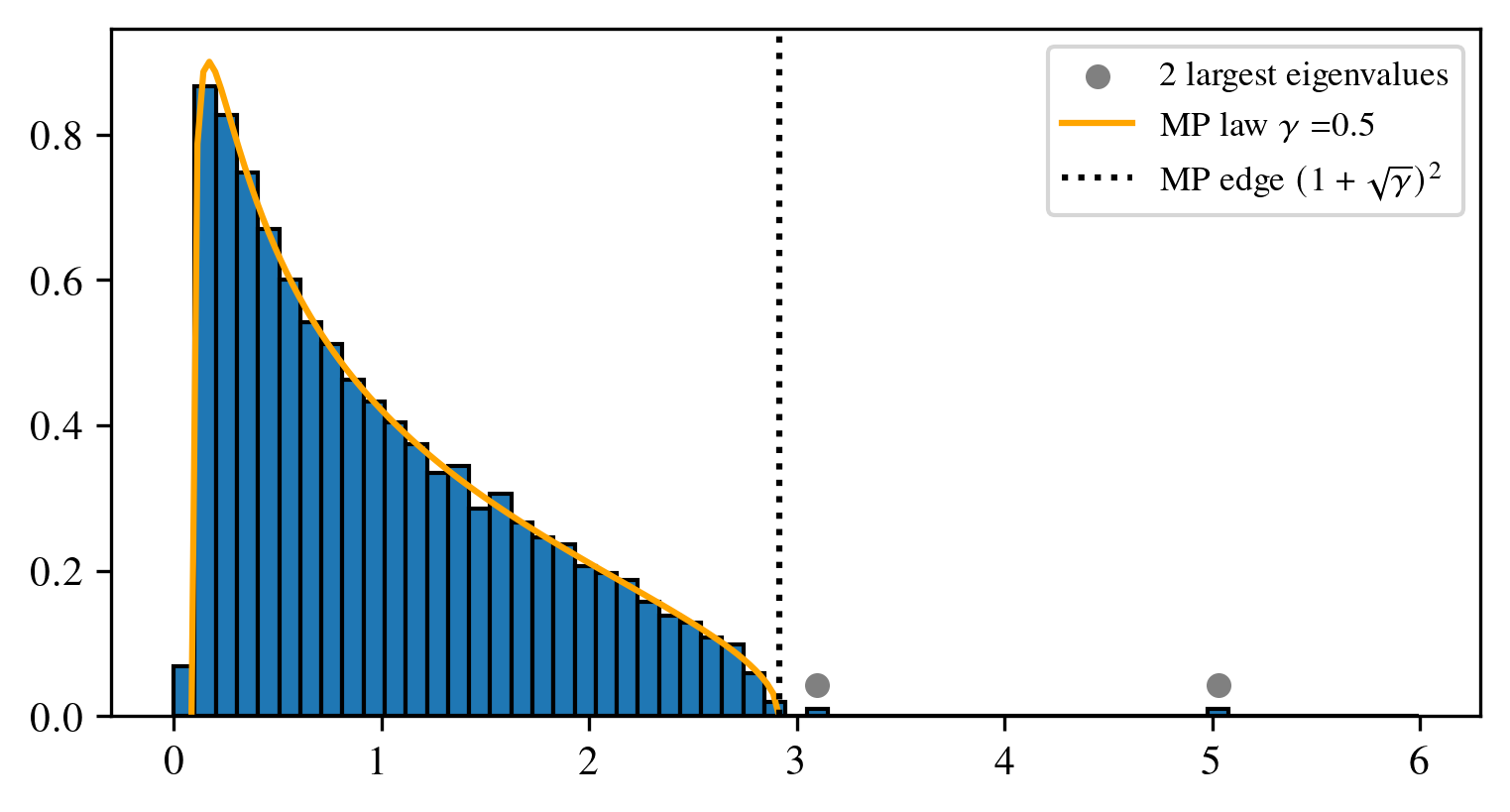}
        \caption{Example of two eigenvalues escaping the bulk of the \acf{MP} distribution. In this example dataset with $n=2000$, $d=1000$, $K=2$, and $\gamma=0.5$, the spike signal strengths were set to $\beta_1 = 4$ and $\beta_2 = 5$, the spike probabilities to $\pi_1 = 0.4$ and $\pi_2 = 0.6$, and the spike correlation to $\theta = 0.4$. This plot shows both the theoretical (orange line) and the empirical (blue bars) \acl{MP} distributions. For these particular hyperparameter values, we observe two eigenvalues (gray circles) successfully popping out of the bulk of the \acl{MP} distribution, reporting detectability.}
        \label{fig:eig-smm}
    \end{subfigure}
    \begin{subfigure}{\textwidth}
        \centering
        \includegraphics[width=0.9\textwidth]{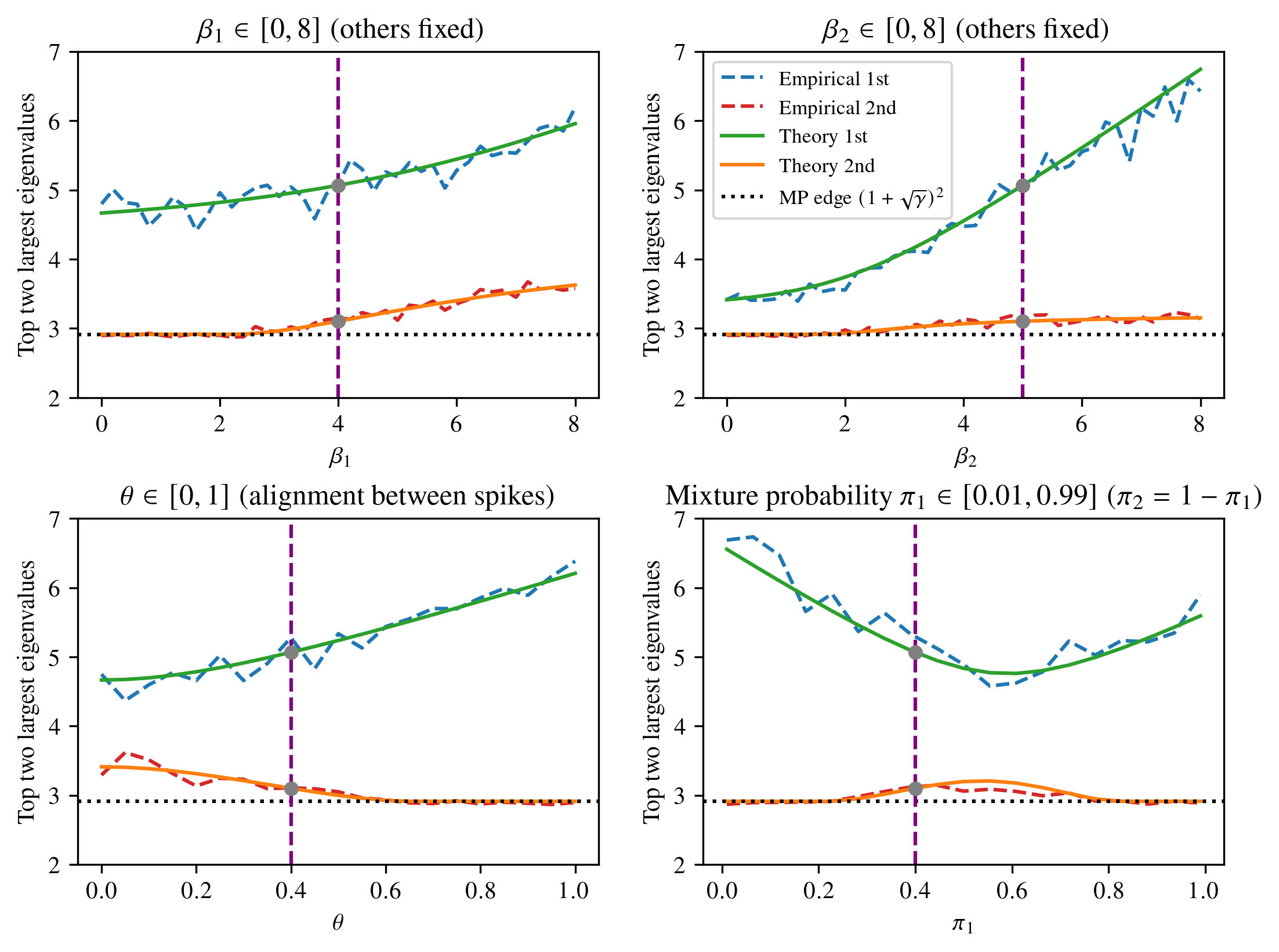}
        \caption{Effect of changing one of the hyperparameters. 
        We show how varying a single hyperparameter, while keeping all others the same as in panel~\ref{fig:eig-smm} impact the two largest eigenvalues. 
        We explore varying the signal strength of the first spike, $\beta_1$ (top-left), the signal strength of the second spike, $\beta_2$ (top-right), the correlation between the spikes, $\theta$ (bottom-left), and the spike probabilities, $\pi_1$ and $\pi_2$ (bottom-right).
        In all plots, the situation of panel~\ref{fig:eig-smm} is shown as a vertical purple line.
        We see that only a subset of all hyperparameter value combinations leads to two eigenvalues escaping the bulk of the \acl{MP} distribution.
        }   
        \label{fig:smm-one-param-change}
    \end{subfigure}
    \caption{Empirical demonstrations of the \ac{SMM} phase transition. 
    Panel~\ref{fig:eig-smm} shows one particular scenario, where two eigenvalues successfully escape the bulk of the \acl{MP} distribution, effectively reporting the spikes as recoverable from the noise.
    Panel~\ref{fig:smm-one-param-change} shows variations on the scenario of panel~\ref{fig:eig-smm}, demonstrating how the interplay between the different \ac{SMM} hyperparameters can result in either detectable or unrecoverable spikes.
    }
\end{figure}

\section{Decomposition of \ac{SMM} covariance matrices \& \ac{SMM} phase transition proof}
\label{sec:cov-matrices+proof}

In section~\ref{sec:decomposition-covariance}, we start by decomposing the covariance matrices that the \ac{SMM} can generate, and then lay out a proof strategy for its phase transition. In sections~\ref{sec:extreme-eig-val-after-K} to \ref{sec:perturbation-analysis}, we prove the necessary facts.

\subsection{Prelude (decomposition of the covariance matrices)}
\label{sec:decomposition-covariance}
Let $z_i \in\{1,\dots,K\}$ be a latent categorical variable indicating which component of the mixture (\textit{i.e.}, spike) was used to generate the observation $\y_i$. Thus, $\y_i | (z_i=k) \sim \mathcal{N}(0,\beta_k \vb_k \vb_k^\mathsf{T} + \I_d)$. With $\x_i \sim \mathcal{N}(0,\I_d)$,
the covariance matrix from the \ac{SMM}~\eqref{eq:model} can be expanded as: 
\begin{align}
    \mathbf{S}_{d,n} = \frac{1}{n} \sum_{i=1}^n \y_i \y_i^\mathsf{T} 
    &= \frac{1}{n}\sum_{k=1}^K (\beta_k \vb_k \vb_k^\mathsf{T} + \I_d)^{1/2} \left(\sum_{i, z_i=k} \x_i \x_i^\mathsf{T} \right) (\beta_k \vb_k \vb_k^\mathsf{T} + \I_d)^{1/2} \nonumber \\
    &= \sum_{k=1}^{K} \frac{n_k}{n} (\beta_k \vb_k \vb_k^\mathsf{T} + \I_d)^{1/2} \Z_k (\beta_k \vb_k \vb_k^\mathsf{T} + \I_d)^{1/2}, \label{eq:covariance}
\end{align}
where $n_k = |\{i\in [n] \: \text{ s.t. }\:  z_i = k\}|$. Conditioned on $n_k$, $ n_k \Z_k = \sum_{i,z_i=k} \x_i \x_i^\mathsf{T}$ follows a Wishart distribution $\mathcal{W}_d(n_k,\I_d)$.
By expanding the square root, we find:
\begin{align*}
    \Sb_{d,n} 
    &= \sum_{k=1}^K \frac{n_k}{n}\left(\I_d + (\sqrt{\beta_k+1}-1)\vb_k \vb_k^\mathsf{T}  \right) \Z_k \left(\I_d + (\sqrt{\beta_k+1}-1)\vb_k \vb_k^\mathsf{T}  \right) \\
    &= \Z + \sum_{k=1}^K \frac{n_k}{n}\left[ c_k \vb_k \vb_k^\mathsf{T} \Z_k + c_k \Z_k \vb_k \vb_k^\mathsf{T} + c_k^2 (\vb_k^\mathsf{T} \Z_k \vb_k) \vb_k \vb_k^\mathsf{T}\right]\\
    &= \Z + \sum_{k=1}^K \frac{n_k}{n} \left[\vb_k \vb_k^\mathsf{T} \Hb_k + \Hb_k \vb_k \vb_k^\mathsf{T}\right],
\end{align*}
with $n\Z = \sum_{k=1}^K n_k \Z_k = \sum_{i=1}^n \mathbf{z}_i \mathbf{z}_i^\mathsf{T} \sim \mathcal{W}_d(n,\I_d)$, $c_k = \sqrt{\beta_k+1}-1$, and $\Hb_k =(c_k \Z_k + \frac{c_k^2}{2}(\vb_k^\mathsf{T} \Z_k \vb_k) \I_d) \in \mathbb{R}^{d\times d}$.
If we define
\begin{align}
    &\Pb_1 = \begin{bmatrix}
        \sqrt{\frac{n_1}{n}} \Hb_1 \vb_1 & \sqrt{\frac{n_1}{n}} \vb_1 & \ldots & \sqrt{\frac{n_K}{n}} \Hb_K \vb_K & \sqrt{\frac{n_K}{n}} \vb_K
    \end{bmatrix} \in \mathbb{R}^{d\times 2K} \label{eq:P1},\\ 
    &\Pb_2 = \begin{bmatrix}
        \sqrt{\frac{n_1}{n}} \vb_1 & \sqrt{\frac{n_1}{n}} \Hb_1 \vb_1 & \ldots & \sqrt{\frac{n_K}{n}} \vb_K & \sqrt{\frac{n_K}{n}} \Hb_K \vb_K
    \end{bmatrix} \in \mathbb{R}^{d\times 2K} \label{eq:P2},
    % &\text{and } \A_i = c_i \Z_i + \frac{c_i^2}{2}\vb_i^\mathsf{T} \Z_i \vb_i \I \nonumber
\end{align}
we can decompose the covariance matrix $\mathbf{S}_{d,n}$ as follows: 
\begin{align}
    \mathbf{S}_{d,n} = \Z + \Pb_1\Pb_2^\mathsf{T}.
    \label{eq:covariance_int}
\end{align}

For a proof of \ac{SMM} phase transition, we first derive the extreme eigenvalues of \eqref{eq:covariance_int} in a manner similar to \cite{benaych2011eigenvalues}. Subsequently, we show that one can replace the $2K\times 2K$ limiting matrix by a simpler $ K\times K$ matrix as described in Theorem~\ref{thm:SMM_phase_transition}. The proof relies on the following facts: 
\begin{enumerate}
    \item Extreme eigenvalues after $K$ tend to the edge of the bulk of the \ac{MP} distribution;
    \label{pt:fact-1}
    \item Extreme eigenvalues that do not escape the bulk of the \ac{MP} distribution must tend to the edge of the bulk;
    \label{pt:fact-2}
    \item Similar to \cite{benaych2011eigenvalues}, we express extreme eigenvalues of $\Sb_{d,n}$ outside the bulk as $z$'s such that a $2K\times 2K$ matrix $\M_{d,n}(z)$ is singular;
    \label{pt:fact-3}
    \item The matrix $\M_{d,n}(z)$ converges almost surely to a matrix $\M(z)$;
    \label{pt:fact-4}
    \item The $z$'s in Fact~\ref{pt:fact-4} such that $\M(z)$ is singular are also the $z$'s such that the matrix $T(z)\mathbf{L}$ has an eigenvalue of $1$, for a $K\times K$ real matrix $\mathbf{L}$; and
    \label{pt:fact-5}
    \item The $z$'s such that $\M_{d,n}(z)$ is singular converge to the $z$'s such that $\M(z)$ is singular (continuity lemma, adapted from \cite{benaych2011eigenvalues}).
    \label{pt:fact-6}
\end{enumerate}
Fact~\ref{pt:fact-1} is proven in section~\ref{sec:extreme-eig-val-after-K}. In section~\ref{sec:extreme-eig-val-before-K}, we prove Facts~\ref{pt:fact-2} and \ref{pt:fact-4}, and recall Fact~\ref{pt:fact-3} from \cite{benaych2011eigenvalues}. Facts~\ref{pt:fact-5} and \ref{pt:fact-6} are proven in section~\ref{sec:schur-complement-equivalence} when $\Lb$ has a simple spectrum. In section~\ref{sec:perturbation-analysis}, we extend this result to the case where $\Lb$ may have equal eigenvalues.

\subsection{Limits of the extreme eigenvalues $\lambda_i(\Sb_{d,n})$ for $i>K$}
\label{sec:extreme-eig-val-after-K}

In this section, we show that the extreme eigenvalues of $\Sb_{d,n}$ after $K$ tend almost surely to the edge of the bulk of the \ac{MP} distribution. The proof relies on the following two lemmas.

% \begin{lemma}
%     \label{eq:negative-eigenvalues}
%     For matrices $\Pb_1$ and $\Pb_2$, defined in \eqref{eq:P1} and \eqref{eq:P2}, the matrix $\Pb = \Pb_1 \Pb_2^\mathsf{T}$ has at most $K$ non-negative eigenvalues.
% \end{lemma}
\begin{lemma}
    \label{eq:negative-eigenvalues}
    For matrices $\Pb_1$ and $\Pb_2$, defined in \eqref{eq:P1} and \eqref{eq:P2}, the matrix $\Pb = \Pb_1 \Pb_2^\mathsf{T}$ has at most $K$ strictly positive eigenvalues and at most $K$ strictly negative eigenvalues.
\end{lemma}
\begin{proof}
    Reordering the columns of $\Pb_1$ and $\Pb_2$, we can decompose $\Pb$ as
    \begin{align*}
        \Pb = \mathbf{X} \mathbf{J} \mathbf{X}^\mathsf{T},\quad \text{with}
    \end{align*}
    \begin{align*}
        \mathbf{X}&= 
        \begin{bmatrix}
            \vb_1 & \ldots & \vb_K & \A \vb_1 & \ldots & \A \vb_K    
        \end{bmatrix} \in \mathbb{R}^{d\times 2K},\quad \text{and} \\
        \mathbf{J} &= 
        \begin{bmatrix}
         0 & \I_K \\
         \I_K & 0
        \end{bmatrix} \in \mathbb{R}^{2K \times 2K}.
    \end{align*}
    Let $W \subseteq \mathbb{R}^{d}$ be a subspace for which $\y^\mathsf{T} \Pb \y >0$ for every non-zero $\y$.
    We start by proving that the mapping $T(\y) = \mathbf{X}^\mathsf{T} \y$ is injective on $W$. 
    For $\y_0 \in W$ such that $T(\y_0) = 0$, we have 
    \begin{align*}
        0 = \y_0^\mathsf{T} \mathbf{X} \mathbf{J} \mathbf{X}^\mathsf{T} \y_0 = \y_0^\mathsf{T} \Pb \y_0.
    \end{align*}
    The definition of $W$ implies that $\y_0 = 0$, proving that the linear mapping $T$ is injective on $W$. Denoting $T|_W$ as the image of $T$ restricted to $W$, this means that
    \begin{align}
        \dim(W) = \dim(T|_W) \label{eq:image-space}.
    \end{align}
    Furthermore, any non zero $\mathbf{z} \in T|_W$ satisfies $\mathbf{z}^\mathsf{T} \mathbf{J} \mathbf{z} > 0$. The definition of $\mathbf{J}$ implies that $\dim(T|_W) \leq K$ and thus by \eqref{eq:image-space} that $\dim(W)\leq K$.\\
    The argument for the negative eigenvalues is analogous.
\end{proof}

\begin{lemma}[Bounded extreme eigenvalues after $K$]
\label{lem:weyl-lemma}
For $\mathbf{S}_{d,n} = \Z+\Pb_1\Pb_2^\mathsf{T}$, defined in \eqref{eq:covariance_int} and assuming that $d\geq 2K$, we have: 
\begin{align*}
    &\lambda_{i+K}(\Z) \leq \lambda_i(\Sb_{d,n}) \leq \lambda_{i-K}(\Z), &\forall K< i \leq d-K.
\end{align*}
\end{lemma}
\begin{proof}
    The proof is based on Weyl's inequality: 
    \begin{align*}
        \lambda_i(\Sb_{d,n}) 
        &\leq \lambda_{i-K}(\Z) + \lambda_{K+1}(\Pb_1\Pb_2^\mathsf{T}) \\
        &\leq \lambda_{i-K}(\Z) &\text{using Lemma}~\ref{eq:negative-eigenvalues}, \text{and}\\
        \lambda_{i+K}(\Z) &\leq \lambda_{i}(\Sb_{d,n}) + \lambda_{K+1}(-\Pb_1\Pb_2^\mathsf{T}) \\
        % &\leq \lambda_i(\Sb_{d,n}) - \lambda_{d}(\Pb_1\Pb_2^\mathsf{T}) \\
        &= \lambda_i(\Sb_{d,n}) 
        &\text{using Lemma}~\ref{eq:negative-eigenvalues}.
    \end{align*}
\end{proof}

Since the infimum and supremum of the bulk of the \acf{MP} distribution $\mu_{\text{MP}(\gamma)}$ correspond to $a= (1-\sqrt{\gamma})^2$ and $b=(1+\sqrt{\gamma})^2$, it follows from \cite{benaych2011eigenvalues} (section 6.2.1) that for every $i\geq 1$: 
\begin{align*}
    \lambda_i(\Z) &\xrightarrow[]{\text{a.s.}} b,  \\
    \lambda_{d-i+1}(\Z) &\xrightarrow[]{\text{a.s.}} a.
\end{align*}
Combining this with Lemma~\ref{lem:weyl-lemma}, we obtain for every $i > K$: 
\begin{align*}
    \lambda_i(\Sb_{d,n}) \xrightarrow[]{\text{a.s.}} b, \\
    \lambda_{d-i+1}(\Sb_{d,n}) \xrightarrow[]{\text{a.s.}} a.
\end{align*}
In other words, all extreme eigenvalues of $\Sb_{d,n}$ beyond the first $K$ almost surely tend to the edge of the bulk of the \ac{MP} distribution.

\subsection{Limits of the extreme eigenvalues $\lambda_i(\Sb_{d,n})$ for $i\leq K$}
\label{sec:extreme-eig-val-before-K}

Using the lower bound from Lemma~\ref{lem:weyl-lemma}, we know for $i \geq 1$ that 
\begin{align}
    \lambda_i(\Sb_{d,n}) \geq \lambda_{i+K}(\Z).
    \label{eq:lower-bound-extreme_Sn}
\end{align}
As in \cite{benaych2011eigenvalues}, the convergence of extreme eigenvalues of $\Z$ together with \eqref{eq:lower-bound-extreme_Sn} implies that
\begin{align*}
    \liminf_{n\to \infty} \lambda_i(\Sb_{d,n}) \geq b.
\end{align*}
This shows that extreme eigenvalues of $\Sb_{d,n}$ that do not escape the bulk of the \ac{MP} distribution must necessarily converge to its edge $b$.

In the following, we are interested in extreme eigenvalues that lie outside the bulk of the \ac{MP} distribution.
Similar to \cite{benaych2011eigenvalues}, we will show that the extreme eigenvalues of $\Sb_{d,n}$ are the $z$'s such that a certain matrix $\mathbf{M}_{d,n}(z)$ is singular.
It is known that the eigenvalues of the matrix $\Z$ weakly converge to the \ac{MP} distribution. Any $z \notin \{\lambda_1(\Z), \ldots, \lambda_n(\Z)\}$ is an eigenvalue of $\Sb_{d,n}$ iff: 
\begin{align*}
    0 &= \det(z \I_d - \Sb_{d,n}) \\
    &= \det(z \I_d-\Z)\det(\I_d - (z \I_d-\Z)^{-1}\Pb_1 \Pb_2^\mathsf{T}) &\text{since } z \notin \{\lambda_1(\Z), \ldots, \lambda_n(\Z)\}, \\
    &= \det(\I_{2K} -\Pb_2^\mathsf{T}(z\I_d-\Z)^{-1}\Pb_1) &\text{using Sylvester's determinant identity.}
\end{align*}
This shows that any $z \notin \{\lambda_1(\Z), \ldots, \lambda_n(\Z)\}$ is an eigenvalue of $\Sb_{d,n}$ if and only if the $2K\times 2K$ matrix $\mathbf{M}_{d,n}(z) = \I_{2K} - \Pb_2^\mathsf{T}(z\I_d-\Z)^{-1}\Pb_1$ is singular.
Since we know that $\lambda_d(\Z)$ and $\lambda_1(\Z)$ converge to the edge of the bulk of the \ac{MP} distribution, the eigenvalues outside the bulk are the $z$'s from the set:
\begin{align}
 \mathcal{K}(\eta) &:= \left\{z\in \mathbb{C}; d(z,[a,b])\geq \eta\right\},   &\forall \eta>0.
 \label{eq:uniform-convergence-set}
\end{align}
Using Theorem~\ref{thm:convergence_invariant_bounded_mat} together with Lemmas~\ref{lem:limit_type1} and \ref{lem:limit_type2}, we get that the matrix $\M_{d,n}(z)$ converges almost surely to a matrix $\M(z)$ as $n,d \to \infty$ such that $\frac{d}{n}\to \gamma$.

\begin{align*}
    {\M_{d,n}(z)}_{2l-1,2m-1} &= \ind_{l=m} - \frac{\sqrt{n_l n_m}}{n} \vb_l^\mathsf{T} (z \I_d-\Z)^{-1} \Hb_m \vb_m \\
    &= \ind_{l=m} - \frac{\sqrt{n_l n_m}}{n}\left[ c_m \vb_l^\mathsf{T} (z \I_d-\Z)^{-1}\Z_m \vb_m + \frac{c_m^2}{2} \vb_l^\mathsf{T}\Z_m\vb_l \vb_l^\mathsf{T} (z \I_d-\Z)^{-1}\vb_m \right]\\
    &\to \ind_{l=m} - \sqrt{\pi_l \pi_m} \theta_{l,m} \left[c_m T(z) + \frac{c_m^2}{2}T(z)(1-\gamma G(z))\right] \\
    &= \ind_{l=m} - \sqrt{\pi_l \pi_m} \theta_{l,m} T(z) c_m \left[1 + \frac{c_m}{2}(1-\gamma G(z))\right],\\
    %%%%%%%%%%%%%%%%%%%%%%%%%%%%%%
    {\M_{d,n}(z)}_{2l-1,2m} &= - \frac{\sqrt{n_l n_m}}{n} \vb_l^\mathsf{T} (z \I_d-\Z)^{-1} \vb_m \\
    &\to -\sqrt{\pi_l \pi_m} \theta_{l,m} G(z),\\
    %%%%%%%%%%%%%%%%%%%%%%%%%%%%%%%
    {\M_{d,n}(z)}_{2l,2m-1} &= - \frac{\sqrt{n_l n_m}}{n} \vb_l^\mathsf{T} \Hb_l (z \I_d-\Z)^{-1} \Hb_m \vb_m \\
    &= -\frac{\sqrt{n_l n_m}}{n} c_l c_m \vb_l^\mathsf{T} \left(\Z_l + \frac{c_l}{2}\vb_l^\mathsf{T} \Z_l \vb_l \I_d \right)(z \I_d-\Z)^{-1}\left(\Z_m + \frac{c_m}{2} \vb_m^\mathsf{T} \Z_m \vb_m \I_d\right) \vb_m \\
    &\to - \sqrt{\pi_l \pi_m} \theta_{l,m} c_l c_m \left[ T(z)\left(\frac{\gamma}{\pi_l}\delta_{l=m} + \frac{1}{1-\gamma G(z)}\right) + \frac{c_l+c_m}{2} T(z) + \frac{c_lc_m}{4}G(z)\right] \\
    &= - \sqrt{\pi_l \pi_m} \left[ \theta_{l,m} c_l c_m \frac{T(z) \gamma }{\pi_l}\delta_{l=m} + \theta_{l,m} c_l c_m \frac{T^2(z)}{G(z)}\left(1 + \frac{c_l}{2}(1-\gamma G(z))\right)\left(1 + \frac{c_m}{2}(1-\gamma G(z))\right)\right]\\
    &\text{using the relation } G(z) = T(z)(1-\gamma G(z)) \text{ in the last step, and}\\
    %%%%%%%%%%%%%%%%%%%%%%%%%%%%%%%
    {\M_{d,n}(z)}_{2l,2m} &= \ind_{l=m} - \frac{\sqrt{n_l n_m}}{n} \vb_l^\mathsf{T} \Hb_l (z \I_d-\Z)^{-1} \vb_m \\
    &= \ind_{l=m} - \frac{\sqrt{n_l n_m}}{n} \left[c_l \vb_l^\mathsf{T} \Z_l(z\I_d-\Z)^{-1} \vb_m + \frac{c_l^2}{2} \vb_l^\mathsf{T} \Z_l\vb_l \vb_l^\mathsf{T} (z \I_d-\Z)^{-1}\vb_m \right]\\
    &\to \ind_{l=m} - \sqrt{\pi_l \pi_m} \theta_{l,m} \left[c_l T(z) + \frac{c_l^2}{2}T(z)(1-\gamma G(z))\right] \\
    &= \ind_{l=m} - \sqrt{\pi_l \pi_m} \theta_{l,m} c_l T(z) \left[1 + \frac{c_l}{2}(1-\gamma G(z))\right],
\end{align*}
with $G(z)$ the Cauchy Transform of the \acl{MP} distribution: 
\begin{align*}
    G(z) = \int \frac{1}{z-t}\text{d}\mu_{\text{MP}(\gamma)}(t).
\end{align*}
Using $a_i = c_i(1+\frac{c_i}{2}(1-\gamma G(z))$, we define the matrices $\A_i, \Bb_{i,j} \in \mathbb{R}^{2 \times 2}$ and the vectors $\cb_i, \db_i \in \mathbb{R}^{2}$ as follows: 
\begin{align*}
    \A_i &:= \begin{bmatrix}
        a_i T(z) & G(z) \\
        c_i^2 T(z)\gamma/\pi_i + a_i^2\frac{T^2(z)}{G(z)} & a_i T(z)
    \end{bmatrix}, \\
    \Bb_{i,j} 
    &:= \begin{bmatrix}
        a_j T(z) & G(z) \\
        a_i a_j T^2(z)/G(z) & a_i T(z)
    \end{bmatrix} = \cb_i \db_j^\mathsf{T}, \\
    \cb_i &:= \begin{bmatrix}
        1 \\
        a_i T(z)/G(z)
    \end{bmatrix} , 
    \db_i := \begin{bmatrix}
        a_i T(z) \\
        G(z)
    \end{bmatrix}.
\end{align*}
With this notation, we obtain that 
$\M_{d,n}(z) \xrightarrow[d,n\to \infty, \frac{d}{n}\to \gamma]{\text{a.s.}}\M(z)$, with
\begin{align}
    \M(z) &:=\I_{2K} - \begin{bmatrix}
          \pi_1 \A_1 & \sqrt{\pi_1 \pi_2} \theta_{1,2} \Bb_{12} & \ldots & \sqrt{\pi_1 \pi_K} \theta_{1,K} \Bb_{1K} \\
         \sqrt{\pi_1 \pi_2} \theta_{2,1} \Bb_{21} & \pi_2 \A_2 & \ldots & \sqrt{\pi_2 \pi_K} \theta_{2,K} \Bb_{2K} \\
         \ldots & \ldots \\
         \sqrt{\pi_1 \pi_K} \theta_{K,1} \Bb_{K,1} & \ldots & \ldots & \pi_K \A_K
    \end{bmatrix} := \I_{2K}-\Lb
    \label{eq:limit_M_matrix}.
\end{align}
Moreover, the convergence of $\M_{d,n}$ is uniform on $\mathcal{K}(\eta)$ for all $\eta>0$, \eqref{eq:uniform-convergence-set}.

\subsection{Equivalence using determinant relation}
\label{sec:schur-complement-equivalence}
In this section, we want to find $z$ such that the matrix $\M(z)$ in \eqref{eq:limit_M_matrix} is singular.
We start by defining the matrix $\tilde{\M}(z) = \I_{2K}-\tilde{\Lb}(z)$ by its entries: 
\begin{align*}
    [\tilde{\M}(z)]_{i,j} := (-1)^{i+j} [\M(z)]_{i,j}.
\end{align*}
By design, the matrix $\tilde{\M}(z)$ has the same determinant as the matrix $\M(z)$ since it is obtained by multiplying all even rows and even columns by $-1$.
We can rewrite the absolute value of the determinant as:
\begin{align}
    \left|\det(\M(z))\right| &= \left|\det(\M(z)\tilde{\M}(z))\right|^{1/2} \nonumber\\
    &= \left|\det\left((\I_{2K} - \Lb(z))(\I_{2K}-\tilde{\Lb}(z))\right)\right|^{1/2} \nonumber\\
    &= \left|\det\left(\I_{2K} - \left(\Lb(z)+\tilde{\Lb}(z) - \Lb(z)\tilde{\Lb}(z)\right)\right)\right|^{1/2}
    \label{eq:final-det-simplifications}.
\end{align}

Defining $\mathbf{W} := \Lb(z)+\tilde{\Lb}(z) - \Lb(z)\tilde{\Lb}(z)$ results through computation (and cancelation) in: 
\begin{align*}
    \mathbf{W}_{[2i:2i+2,2i:2i+2]}
    &= \pi_i \beta_i T(z) \begin{bmatrix}
        1 & 0 \\
        0 & 1
    \end{bmatrix}, &i\in[0,K-1],\\ 
    \mathbf{W}_{[2i:2i+2,2j:2j+2]}
    &= \sqrt{\pi_i \pi_j} \theta_{i,j}\begin{bmatrix}
       \beta_j T(z)  & 0 \\
        T^{2}(z)(c_i^2a_j - c_j^2a_i) & \beta_i T(z)
    \end{bmatrix}, &i,j \in [0,K-1], i\neq j.\\ 
\end{align*}
We finish by linking the matrix $\mathbf{W} \in \mathbb{C}^{2K\times 2K}$ to the matrix $\Lb\in \mathbb{R}^{K\times K}$ defined in Theorem~\eqref{thm:SMM_phase_transition}.
Let 
\begin{align*}
    \mathbf{R} := \text{Diag}\left(\sqrt{\beta_1},\sqrt{\beta_1},\sqrt{\beta_2},\sqrt{\beta_2}, \ldots, \sqrt{\beta_K},\sqrt{\beta_K}\right) \in \mathbb{R}^{2K\times 2K}
\end{align*} 
be a rescaling matrix.
Furthermore, let $\mathbf{Q} \in \mathbb{R}^{2K\times 2K}$ be a permutation matrix such that multiplication with $\mathbf{Q}$ on the left moves all even rows to the first $K$ rows and multiplication with $\mathbf{Q}^\mathsf{T}$ on the right moves all even columns to the first $K$ columns. More explicitly, with the permutation $\tilde{\pi}$ defined as 
\begin{align*}
    \tilde{\pi}(i) = 
    \begin{cases}
        2i &\text{if } i \in [0,K-1] \\
        2(i-K) +1 &\text{if } i \in [K , 2K-1]
    \end{cases}
    ,
\end{align*}
the permutation matrix $\mathbf{Q}$ can be defined as
\begin{align*}
    \mathbf{Q}_{i,j} 
    &= \begin{cases}
        1 & \text{if } j = \tilde{\pi}(i) \\
        0 & \text{otherwise}
    \end{cases},
\end{align*}
or equivalently:
\begin{align*}
    \mathbf{Q} =
    \begin{bmatrix}
        1 & 0 & 0 & 0 & \ldots & 0 & 0\\
        0 & 0 & 1 & 0 & \ldots & 0 & 0\\
        &&&\ldots \\
        0 & 0 & 0 & 0 & \ldots & 1 & 0\\
        0 & 1 & 0 & 0 & \ldots & 0 & 0 \\
        0 & 0 & 0 & 1 & \ldots & 0 & 0 \\
        &&& \ldots   \\
        0 & 0 & 0 & 0 & \ldots & 0 & 1 
    \end{bmatrix}.
\end{align*}
The matrix $\mathbf{Q}\mathbf{R} \mathbf{W} \mathbf{R^{-1}\mathbf{Q^\mathsf{T}}}$ is then block lower triangular with all diagonal blocks equal and corresponding to $T(z)\Lb$.
Knowing this, equation~\eqref{eq:final-det-simplifications} can then be simplified to:
\begin{align}
    \left|\det(\M(z))\right| &= \left|\det(\I_{2K} - \mathbf{W})\right|^{1/2} \nonumber\\
    &=\left|\det\left(\I_{2K} - \mathbf{Q}\mathbf{R} \mathbf{W} \mathbf{R^{-1}\mathbf{Q^\mathsf{T}}}\right)\right|^{1/2} \nonumber\\
    &= \left|\det(\I_{K}-T(z)\Lb)^2\right|^{1/2} \nonumber\\
    &= \left|\det(\I_{K}-T(z)\Lb)\right| \nonumber\\
    &= \prod_{i=1}^K \left|1-T(z)\lambda_i(\Lb)\right|.
    \label{eq:det_final_equation}
\end{align}
The phase transition for the extreme eigenvalues emerges since the \ac{MP} distribution is compactly supported on $[a,b]$, and thus $T(z)$ is defined outside $[a,b]$. 
Let
\begin{align*}
    T(a^-) &:= \lim_{z \to a} T(z) = -1/\sqrt{\gamma}, \text{and} \\
    T(b^+) &:= \lim_{z \to b} T(z) = 1/\sqrt{\gamma}
\end{align*}
with $T(z)$ being a homeomorphism from $(-\infty, a)$ to $(T(a^-),0)$ and from $(b,\infty)$ to $(0,T(b^+))$.
Since $\lambda_i(\Lb)\geq 0$, there exists a solution to $T(z) = \frac{1}{\lambda_i(\Lb)}$ only when $\frac{1}{\lambda_i(\Lb)} < T(b^+)$, which is equivalent to $\lambda_i(\Lb) > \sqrt{\gamma}$.\\

When the eigenvalues of the limiting matrix $\Lb$ are distinct, the conclusion of Theorem~\ref{thm:SMM_phase_transition} follows from an application of
Lemma~\ref{lem:convergence_lemma}. 
When the eigenvalues are not distinct, we show in section~\ref{sec:perturbation-analysis} that, using a small perturbation, the conclusions of Theorem~\ref{thm:SMM_phase_transition} remain valid.

\begin{lemma}[Continuity lemma, adapted from Lemma 6.1 in \cite{benaych2011eigenvalues}] 
\label{lem:convergence_lemma}
    For $\lambda_1(\Lb) > \ldots > \lambda_r(\Lb) > \sqrt{\gamma}$ with $r\leq K$, there exists $r$ complex sequences, ${z_1}{(d,n)}, \ldots, {z_r}{(d,n)}$, such that $\operatorname{Re}({z_1}{(d,n)}) \geq \ldots \geq \operatorname{Re}({z_r}{(d,n)})$
    converge respectively to the eigenvalues outside the bulk of the \ac{MP} distribution:
    \begin{align*}
        T^{-1}(1/\lambda_1(\Lb)), \ldots ,T^{-1}(1/\lambda_r(\Lb)).
    \end{align*}
\end{lemma}
\begin{proof}
    We start by proving that the $z$'s we are looking for do not escape the bulk of the \ac{MP} distribution at infinity and that they lie in a compact set.
    Since $T(z) \xrightarrow[]{|z|\to \infty} 0$, there exists a real scalar $R \geq (1+\sqrt{\gamma})^2$ such that for every $z\in \mathbb{C},|z|\geq R$: 
    \begin{align*}
        |T(z)|\leq \frac{1}{2\lambda_1(\Lb)}.
    \end{align*}
    As such, for all $i\in [K], |z|\geq R$ we have:
    \begin{align*}
        |1 - T(z)\lambda_i(\Lb)| \geq \frac{1}{2}.
    \end{align*}
    Combining this with \eqref{eq:det_final_equation}, we get for $|z|\geq R$ that
    \begin{align}
        |\det\left(\M(z)\right)|\geq 2^{-K}.
        \label{eq:lower_bound_det}
    \end{align}
    Moreover, we have shown in section~\ref{sec:extreme-eig-val-before-K} that $\M_{d,n}(z)$ converges uniformly on $\mathcal{K}(\eta)=\left\{z \in \mathbb{C};\quad d(z,[a,b])\geq \eta\right\}$ for all $\eta >0$.
    Together with \eqref{eq:lower_bound_det}, this uniform convergence implies that for a large enough $n$ and $d$, the $z$'s such that $\M_{d,n}(z)$ is singular are supported on a compact subset $\mathcal{K}_1 \subset \mathcal{K}(\eta)$. This ensures that $z$ does not run away to infinity.

    For $l\neq r \in \mathcal{K}_1$, we denote $\text{Card}_{l,r}$ as the number of $z$'s in $B(l,r)$, the complex ball of diameter $[l,r]$, such that $\det(\M(z))=0$. Similarly, $\text{Card}_{l,r}(d,n)$ denotes the numbers of $z$'s such that $\det(\M_{d,n}(z))=0$.
    With $l,r\notin \{\lambda_1(\Lb),\ldots, \lambda_r(\Lb)\}$ ,$\gamma$ is a circle of diameter $[l,r]$.
    We finish by showing that for every $l,r$, $\text{Card}_{l,r}(d,n)$ converges almost surely to $\text{Card}_{l,r}$:
    \begin{align*}
        \text{Card}_{l,r} = \frac{1}{2\pi i}\int_\gamma\frac{\partial_z \det \M(z)}{\det \M(z)} dz = \lim_{\substack{n\to\infty,\\d\to\infty,\\\frac{d}{n}\to\gamma}} \frac{1}{2\pi i}\int_\gamma\frac{\partial_z \det \M_{d,n}(z)}{\det \M_{d,n}(z)} dz.
    \end{align*}
    This shows that for $d$ and $n$ large enough, there exist $r$ distinct solutions to $\det\left( \M_{d,n}(z)\right)=0$, denoted as $z_1(d,n), \ldots, z_r(d,n)$ and with $\operatorname{Re}({z_1}{(d,n)}) \geq \ldots \geq \operatorname{Re}({z_r}{(d,n)})$. Furthermore, since we can choose $l$ and $r$ arbitrarily close to each other, we must have
    \begin{align*}
        z_{i}(d,n) &\xrightarrow[\substack{n\to\infty,\\d\to\infty,\\\frac{d}{n}\to\gamma}]{\mathrm{a.s.}} T^{-1}\left(\frac{1}{\lambda_i(\Lb)}\right), & \forall i \in [1,r].
    \end{align*}
\end{proof}

\subsection{Perturbation analysis}
\label{sec:perturbation-analysis}

The purpose of this section is to extend Theorem~\ref{thm:SMM_phase_transition}
to the case where $\Lb$ has repeated eigenvalues.
We do so by introducing an arbitrarily small perturbation of the spikes $\vb_1, \ldots, \vb_K$ that yields a limiting matrix with a simple spectrum, by subsequently applying the results from section~\ref{sec:schur-complement-equivalence}, and finally by removing the perturbation by continuity.

Let $\wb_1, \ldots, \wb_K \sim \mathcal{N}(0,\frac{1}{d}\I_d)$ be $K$ independent vectors sampled from a multivariate Gaussian distribution.
For arbitrary real scalars $t_1, \ldots, t_K \in \mathbb{R}$, we construct perturbed spike vectors $\tilde{\vb}_i$ as follows:
\begin{align}
    \tilde{\vb}_i &:= \frac{\vb_i + t_i \wb_i}{\|\vb_i + t_i \wb_i\|}, &\forall i \in [K], \label{eq:perturbated_v}\\
    \tilde{\beta}_i &:= (1+t_i^2) \beta_i, &\forall i \in [K]. \label{eq:perturbated_beta}
\end{align}
For $i\neq j$, we get
\begin{align*}
    \tilde{\vb}_i^\mathsf{T}\tilde{\vb}_j \xrightarrow[d\to \infty]{\text{a.s.}} \frac{\theta_{i,j}}{\sqrt{(1+t_i^2)(1+t_j^2)}}:= \tilde{\theta}_{i,j},
\end{align*}
and 
\begin{align}
    \vb_i^\mathsf{T}\tilde{\vb_i} \xrightarrow[d\to\infty]{\text{a.s.}} \frac{1}{\sqrt{1+t_i^2}}.
    \label{eq:perturbation-corellation}
\end{align}
With $\tilde{\Db} = \text{Diag}(\pi_1\tilde{\beta_1}, \ldots, \pi_K\tilde{\beta}_K)$ and $\tb = (t_1,\ldots, t_K)$, we define
\begin{align*}
    \tilde{\Lb}(\tb) := \tilde{\Db}^{1/2} \tilde{\mathbf{\theta}} \tilde{\Db}^{1/2}.
\end{align*}
The discriminant $\Delta(\tb)$ of the characteristic polynomial of $\tilde{\Lb}(\tb)$ is a real analytic function of the entries of $\tilde{\Lb}(\tb)$.
As $\tilde{\Lb}(\tb)$ depends on $\tb$ only in the diagonal, using the strengthened form of the Gershgorin circle theorem, we can choose $\tb$ such that the Gershgorin discs do not intersect. Hence, $\tilde L(\tb)$ has a simple spectrum and
$\Delta(\tb)\neq 0$. Therefore, $\Delta$ is not identically zero.
Since the zero set of a not identically zero, real analytic function has an empty interior, the set of perturbations yielding a simple spectrum is dense.
Moreover, since $\tilde{\Lb}(\tb) \to \Lb$ as $\|\tb\|_{\infty} \to 0$,
for every real scalar $\delta_1>0$, there exists an arbitrarily small $\tb_0$ ($\|\tb_0\|_{\infty} \leq \delta_2$, with $\delta_2$ a real scalar) such that $\tilde{\Lb}(\tb_0)$ has a simple spectrum and
\begin{align}
    \|\Lb-\tilde{\Lb}(\tb_0)\|_{\mathrm{op}} &\leq \delta_1.
    \label{eq:perturbation-L-spectral}
\end{align}
Let $\tilde{\Sb}_{d,n}$ be the covariance matrix associated with the perturbed vectors from \eqref{eq:perturbated_v} and \eqref{eq:perturbated_beta} with $\tb_0$.
Furthermore, let
\begin{align*}
    \rho(l) :=
    \begin{cases}
                T^{-1}\left(1/l\right) &\text{if } l > \sqrt{\gamma} \\
                (1+\sqrt{\gamma})^{2} & \text{otherwise }
    \end{cases}
    .
\end{align*}
Since $\tilde{\Lb}(\tb_0)$ has distinct eigenvalues, the conclusion from section~\ref{sec:schur-complement-equivalence} shows that
\begin{align*}
    \lambda_i(\tilde{\Sb}_{d,n}) \xrightarrow[\substack{n\to\infty,\\d\to\infty,\\\frac{d}{n}\to\gamma}]{\text{a.s.}} \tilde{\rho}_i := \rho\left(\lambda_i(\tilde{\Lb}(t_0))\right).
\end{align*}
Using the triangle inequality, we get
\begin{align}
\left| \lambda_i(\Sb_{d,n}) - \rho(\lambda_i(\Lb)) \right|
&\leq
\underbrace{\left| \lambda_i(\Sb_{d,n})-\lambda_i(\tilde{\Sb}_{d,n}) \right|}_{\text{(A)}}
+
\underbrace{\left|\lambda_i(\tilde{\Sb}_{d,n}) - \tilde{\rho_i}\right|}_{\text{(B)}}
+
\underbrace{\left|\tilde{\rho_i} - \rho(\lambda_i(\Lb)) \right|}_{\text{(C)}} .
\label{eq:pertubation-3-terms}
\end{align}
We show the following facts:
\begin{itemize}
    \item[(C)] is small by continuity of the function $\rho$;
    \item[(B)] tends to zero by a previously established result in the case 
    where the limiting matrix has distinct eigenvalues; and
    \item[(A)] is small because the perturbation of the spikes is small. Hence, 
    $\|\Sb_{d,n} - \tilde{\Sb}_{d,n}\|_{\mathrm{op}}$ can be made arbitrarily small and Weyl's inequality applies.
\end{itemize}
In the following, we let the real scalar $\varepsilon>0$ be arbitrary.
Starting with the (C) term, using Weyl's inequality together with \eqref{eq:perturbation-L-spectral}, we get that
\begin{align*}
    |\lambda_i(\tilde{\Lb}(t_0)) - \lambda_i(\Lb)| &\leq \delta_1.
\end{align*}
By selecting $\delta_1$ sufficiently small, we get by continuity of the function $\rho$ that
\begin{align}
    |\tilde{\rho_i} - \rho(\lambda_i(\Lb))|\leq \frac{\varepsilon}{3}
    \label{eq:perturbation-bound-C}.
\end{align}
For the (B) term, the results from section~\ref{sec:schur-complement-equivalence} imply that for a large enough $n$ and $d$, almost surely
\begin{align}
    |\lambda_i(\tilde{\Sb}_{d,n}) - \tilde{\rho_i}| \leq \frac{\varepsilon}{3}.
    \label{eq:perturbation-bound-B}
\end{align}
We finish by bounding the (A) term, using Weyl's inequality once again:
\begin{align*}
    |\lambda_i(\Sb_{d,n})-\lambda_i(\tilde{\Sb}_{d,n})| 
    &\leq
    \|\Sb_{d,n} - \tilde{\Sb}_{d,n}\|_{\mathrm{op}} .
\end{align*}
Combining this with the definitions of $\Sb_{d,n}$ and $\tilde{\Sb}_{d,n}$ in \eqref{eq:covariance} and the triangle inequality, we obtain:
\begin{align*}
    &\left|\lambda_i(\Sb_{d,n})-\lambda_i(\tilde{\Sb}_{d,n})\right|  \\
    &\leq \sum_{k=1}^K \frac{n_k}{n} \left\|\left(\I_d + \beta_k \vb_k\vb_k^\mathsf{T}\right)^{1/2}\Z_k\left(\I_d + \beta_k \vb_k\vb_k^\mathsf{T}\right)^{1/2} - 
    \left(\I_d + \tilde{\beta}_k \tilde{\vb}_k\tilde{\vb}_k^\mathsf{T}\right)^{1/2}\Z_k\left(\I_d + \tilde{\beta}_k \tilde{\vb}_k \tilde{\vb}_k^\mathsf{T}\right)^{1/2}\right\|_{\mathrm{op}} \\
            &\leq \sum_{k=1}^K \frac{n_k}{n} \left\|\left(\I_d+\beta_k \vb_k\vb_k^\mathsf{T}\right)^{1/2} - \left(\I_d+\tilde{\beta}_k \tilde{\vb}_k\tilde{\vb}_k^\mathsf{T}\right)^{1/2}\right\|_{\mathrm{op}} \left\|\Z_k\right\|_{\mathrm{op}}\left((1+\beta_k)^{1/2}+(1+\tilde{\beta}_k)^{1/2}\right).
\end{align*}
We proceed by showing that $\left\|\left(\I_d+\beta_k \vb_k\vb_k^\mathsf{T}\right)^{1/2} - \left(\I_d+\tilde{\beta}_k \tilde{\vb}_k\tilde{\vb}_k^\mathsf{T}\right)^{1/2}\right\|_{\mathrm{op}}$ can be made arbitrarily small by an appropriate choice of $\delta_1$:
\begin{align*}
    &\left\|\left(\I_d+\beta_k \vb_k\vb_k^\mathsf{T}\right)^{1/2} - \left(\I_d+\tilde{\beta}_k \tilde{\vb}_k\tilde{\vb}_k^\mathsf{T}\right)^{1/2}\right\|_{\mathrm{op}} \\
    &= \left\|\left(\sqrt{\beta_k+1}-1\right) \vb_k \vb_k^\mathsf{T} - \left(\sqrt{\tilde{\beta}_k+1}-1\right) \tilde{\vb}_k\tilde{\vb}_k^\mathsf{T}\right\|_{\mathrm{op}}\\
    &\leq \left\|\left(\sqrt{\beta_k+1}-1\right) \vb_k \vb_k^\mathsf{T} - \left(\sqrt{\tilde{\beta}_k+1}-1\right) \vb_k\vb_k^\mathsf{T}\right\|_{\mathrm{op}} + \left|\sqrt{\tilde{\beta}_k+1}-1\right|\left\|\vb_k\vb_k^\mathsf{T} - \tilde{\vb}_k \tilde{\vb}_k^\mathsf{T}\right\|_{\mathrm{op}} \\
    &\leq \left|\sqrt{{\beta}_k+1} - \sqrt{\tilde{\beta}_k+1}\right| + \left|\sqrt{\tilde{\beta}_k+1}-1\right| \sqrt{2}\sqrt{1-(\vb_k^\mathsf{T}\tilde{\vb}_k)^2}.
\end{align*}
Plugging in the definition of $\tilde{\beta}_k$ \eqref{eq:perturbated_beta} and the asymptotic correlation of $\vb_k^\mathsf{T}\tilde{\vb}_k$ \eqref{eq:perturbation-corellation}, for an appropriate choice of $\delta_2$, we get that for a large enough $n$ and $d$ almost surely
\begin{align}
    \left|\lambda_i(\Sb_{d,n})-\lambda_i(\tilde{\Sb}_{d,n})\right| &\leq \frac{\varepsilon}{3}.
    \label{eq:perturbation-bound-A}
\end{align}
Finally, using \eqref{eq:perturbation-bound-C},  \eqref{eq:perturbation-bound-B}, and \eqref{eq:perturbation-bound-A} in \eqref{eq:pertubation-3-terms}, we get that for a large enough $n$ and $d$, almost surely
\begin{align*}
    \left|\lambda_i(\Sb_{d,n}) - \rho(\lambda_i(\Lb))\right| &\leq \varepsilon.
\end{align*}

\section{Conclusion}
This work examines the recovery of signals from noisy measurements in high-dimensional regimes.
In doing so, we extend beyond previous work on finitely many, low-rank perturbations of large random matrices.
Our results generalize the phase transition behavior known for such perturbations, moving from a single-spike framework to a multi-spike mixture model setting.
The latter facilitates the study of signal recovery in scenarios where multiple signals underlie a measurement set, which opens up a broad field of new applications.
To our knowledge, this is the first study to characterize a multi-spike mixture scenario and to demonstrate the interaction of correlation, signal energy, and mixture probability in the phase transition.
Consequently, our study builds on and extends prior work on the single-spike model's phase transition.

Specifically, our work examines the feasibility of signal recovery by extreme eigenvalues for measurements that can be modeled by a \acf{SMM}.
The \ac{SMM} is not the first model to facilitate multiple spikes, with previous work, \textit{e.g.}, on linear multi-spike models \cite{arous2024high}.
However, linear formulations do not always suit real-world measurements, particularly in technologies where the mixing of signals is not linear or deviates substantially from linear in certain scenarios.
In such situations, the \ac{SMM} can be essential: it accommodates multiple underlying signals yet does not require them to mix linearly, but instead selects the dominant signal per observation.
Examples of such application areas can be found in our previous work on the \ac{SMM} \cite{SMM_IEEE}, where this model was successfully used in data types ranging from \ac{IMS} in biomedicine to \ac{HSI} in computer vision.

Our findings show that in this multi-spike mixture model setting, the phase transition, and thus signal recovery by extreme eigenvalues, depends on several interacting factors: the correlation between spikes (\textit{i.e.}, how similar in content two underlying signals are), the energy parameter $\beta_k$ (\textit{i.e.}, the absolute strength of each underlying signal), and the mixture probabilities (\textit{i.e.}, how likely it is to encounter each underlying signal). 
Each of these parameters can individually impact the phase transition, and thus our ability to recover signals.
However, the multi-spike setting becomes particularly interesting when we consider the hereto less-studied role of inter-spike-correlation and the interplay between the different parameters.
Examples of this include strongly correlating spikes boosting mutual detectability despite making individual spike recognition harder, and the correlation between spikes substantially modulating the traditional role that signal strength plays in detectability.
More broadly, our results highlight how the parameters of the mixture model interact and jointly shape the phase transition behaviour. Besides the methodological insights and the generalization of single-spike scenarios, we believe that understanding this interplay can become an essential tool for driving experimental design, \textit{e.g.}, when developing physical \ac{IMS} experiments in analytical chemistry and the life sciences.

\section*{Acknowledgments}
Research reported in this publication was supported by the National Institutes of Health (NIH)’s Common Fund, National Institute Of Diabetes And Digestive And Kidney Diseases (NIDDK), and the Office Of The Director (OD) under Award Numbers U54DK120058, U54DK134302, and U01DK133766 (R.V.), by NIH’s Common Fund, National Eye Institute, and the Office Of The Director (OD) under Award Number U54EY032442 (R.V.), by NIH’s National Institute Of Allergy And Infectious Diseases (NIAID) under Award Numbers R01AI138581 and R01AI145992 (R.V.), by NIH’s National Institute On Aging (NIA) under Award Number R01AG078803 (R.V.), and by NIH's National Cancer Institute (NCI) under Award Number U01CA294527 (R.V.). The content is solely the responsibility of the authors and does not necessarily represent the official views of the National Institutes of Health.\\
The authors would like to thank Antoine Maillard (INRIA Paris \& Département d'Informatique, École Normale Supérieure, Université PSL) for fruitful discussions and valuable suggestions.

{\appendices
\newpage
\section{Concentration bounds}
\label{appendixA}

\begin{theorem}[Concentration of Lipschitz functions on the sphere \cite{vershynin2018high} (section 5.1.2)]
\label{thm:lipschitz_sphere}
    Consider a random vector $\y \sim \textnormal{Unif}(\mathbb{S}^{d-1})$ and a Lipschitz function $f$: $\mathbb{S}^{d-1} \to \mathbb{R}$, with Lipschitz constant $\|f\|_{\Lip}$. For every $t\geq 0$ and $d$ large enough, 
    \begin{align*}
        \mathbb{P}\{|f(\y) - \E f(\y)|\geq t\} \leq 2 \exp\left(- \frac{cdt^2}{\|f\|^2_{\Lip}}\right),
    \end{align*}
    for an absolute positive constant $c$.
\end{theorem}

\begin{theorem}[Convergence orthogonal invariant bounded matrix]
\label{thm:convergence_invariant_bounded_mat}
Consider an orthogonal invariant random matrix ensemble given by the probability measure $\mathbb{P}_{d}$ on sets $\Omega_d$ of $d \times d$ complex matrices.
Let $\A \in \mathbb{C}^{d \times d}$ be a matrix sampled from such a model.
Assume:
\begin{enumerate}
    \item[\textnormal{(i)}] \emph{(operator norm bound in probability)} 
    There exists a deterministic function $C(d)$ with 
    \begin{align*}
        C(d) \xrightarrow[d\to\infty]{} C < \infty,
    \end{align*}
    such that
    \begin{align*}
        \sum_{d\geq 1} \mathbb{P}\left\{\|\A\|_{\mathrm{op}} \geq C(d)\right\} < \infty.
    \end{align*}
    \item[\textnormal{(ii)}] \emph{(convergence of the normalized trace)}  
    The first moment converges almost surely:
    \begin{align*}
        \frac{\Tr(\A)}{d} \xrightarrow[d \to \infty ]{\mathrm{a.s.}} \mu.
    \end{align*}
\end{enumerate}
Then, for any pair of random unit vectors $\ub, \vb \in \mathbb{S}^{d-1}$ (independent of $\A$) such that 
    \begin{align*}
        \ub^\mathsf{T} \vb \xrightarrow[d \to \infty]{\mathrm{a.s.}} \ell ,
    \end{align*}
    we have
    \begin{align*}
        \ub^\mathsf{T} \A \vb \xrightarrow[d \to \infty]{\mathrm{a.s.}} \ell \:\mu.
    \end{align*}
\end{theorem}
\begin{proof}
    Projecting $\vb$ on $\ub$, we get the decomposition
    \begin{align*}
     \vb = \frac{(\ub \cdot \vb)}{\|\ub\|} \frac{\ub}{\|\ub\|} + \left(\vb - \frac{(\ub \cdot \vb )}{\|\ub\|} \frac{\ub}{\|\ub\|} \right)    ,
    \end{align*}
    and thus $\ub^\mathsf{T} \A \vb = \A_1 + \A_2$ with 
    \begin{align*}
        \A_1 &= (\ub \cdot \vb) \ub^\mathsf{T} \A \ub ,\\
        \A_2 &= \ub^\mathsf{T} \A (\vb - (\ub \cdot \vb) \ub).
    \end{align*}
    Since $\A$ is orthogonal invariant, $\ub^\mathsf{T} \A \ub$ has the same distribution as $\y^\mathsf{T} \A \y$ for any $\y \in \mathbb{S}^{d-1}$, which we denote as 
    \begin{align*}
        \A_1 \sim (\ub \cdot \vb) \y^\mathsf{T} \A \y.
    \end{align*}
    In particular, we take $\y \sim \text{Unif}(\mathbb{S}^{d-1})$. Similarly, $\A_2 \sim \sqrt{1 - (\ub \cdot \vb)^2}\y_1^\mathsf{T} \A \y_2$ for $\y_1 = (a_1 , \ldots , a_d),\: \y_2 = (b_1 , \ldots b_d)$, the first two rows of a Haar distributed random orthogonal matrix. 
    We then analyze the limits of $\A_1$ and $\A_2$ separately.\\
    
    \noindent$\bullet$ \textbf{Step 1}:
    $\ub^\mathsf{T} \A \ub$ behaves like $g_1(\y) := \y^\mathsf{T} \A \y$ for $\y \sim \text{Unif}(\mathbb{S}^{d-1})$. If we condition on $\A$, $g_1$ is Lipschitz on $\mathbb{S}^{d-1}$ since:
    \begin{align*}
        \nabla g_1(\y) &= (\A+\A^\mathsf{T})\y \\
        \|\nabla g_1(\y)\|_2 &\leq 2 \|\A\|_{\mathrm{op}}.
    \end{align*}
    Furthermore, the conditional expectation gives
    \begin{align*}
        \E \left[g(\y) | \A\right] = \frac{1}{d} \Tr(\A).
    \end{align*}
    Applying Theorem~\ref{thm:lipschitz_sphere}, we find the conditional probability bound: 
    \begin{align}
        \mathbb{P}\left\{ \left|\y^\mathsf{T} \A \y - \frac{1}{d}\Tr(\A) \right| \geq t \:| \A\right\} \leq 2 \exp \left( - \frac{cdt^2}{4 \|\A\|_{\mathrm{op}}^2}\right).
        \label{eq:norm-conditioned-on-A}
    \end{align}
    Let $D := \left\{\big|\y^\mathsf{T} \A \y - \frac{1}{d}\Tr(\A) \big| \geq t\right\}$ and $B := \left\{\|\A\|_{\mathrm{op}} \leq C(d)\right\}$, noting that $B$ is determined by $\A$, we get:
    \begin{align*}
        \mathbb{P}\left\{D|B\right\} = \frac{\E\left\{\ind_D\ind_B\right\}}{\mathbb{P}\left\{B\right\}} = \frac{\E_A\E_D\left\{\ind_D\ind_B|A\right\}}{\mathbb{P}\left\{B\right\}} = \frac{\E_A\ind_B\mathbb{P}\left\{D|A\right\}}{\mathbb{P}\left\{B\right\}}.
    \end{align*}
    Applying the bound \eqref{eq:norm-conditioned-on-A} pointwise gives:  
    \begin{align*}
        \mathbb{P}\left\{\left| \y^\mathsf{T} \A \y - \frac{1}{d}\Tr(\A) \right|\geq t \:|\: \|\A\|_{\mathrm{op}}\leq C(d) \right\} \leq 2 \exp \left(- \frac{cdt^2}{4 C(d)^2}\right).
    \end{align*}
    The assumptions on $\A$ give: 
    \begin{align}
        \mathbb{P}\left\{\left|\y^\mathsf{T}\A\y - \frac{1}{d}\Tr(\A)\right| \geq t\right\} \leq \exp \left(-\frac{c_2dt^2}{C(d)^2}\right) +\mathbb{P}\left\{\|\A\|_{\mathrm{op}} > C(d)\right\}. \label{eq:deviation_trace}
    \end{align}
    We continue by taking $t = t(d)$, a decreasing function, such that the exponential term in \eqref{eq:deviation_trace} is summable. For example with $t=d^{-1/4}$, for $d$ large enough we get 
    \begin{align}
        \mathbb{P}\left\{\left|\y^\mathsf{T}\A\y - \frac{1}{d}\Tr(\A)\right| \geq d^{-1/4}\right\} = \exp \left(-c_3d^{1/2}\right) +\mathbb{P}\left\{\|\A\|_{\mathrm{op}} > C(d)\right\} .\label{eq:summable_RHS}
    \end{align}
    By assumption (i), the right hand side of \eqref{eq:summable_RHS} is summable in $d$. Therefore, the Borel-Cantelli lemma implies: 
    \begin{align*}
        \left|\y^\mathsf{T} \A \y - \frac{1}{d}\Tr(\A)\right| \xrightarrow[d\to \infty]{\mathrm{a.s.}} 0.
    \end{align*}
    Using the assumption of the almost sure convergence of the first moment, we get that: 
    \begin{align*}
        \y^\mathsf{T} \A \y \xrightarrow[d\to \infty]{\mathrm{a.s.}} \mu.
    \end{align*}\\
    
    \noindent$\bullet$ \textbf{Step 2:} Analogous to step 1, we use the function $g_2(\y_1,\y_2) = \y_1^\mathsf{T} \A \y_2$, which is also Lipschitz on $\mathbb{S}^{d-1}\times \mathbb{S}^{d-1}$ since: 
    \begin{align*}
        &\nabla_{\x_1}g_2(\x_1,\x_2) = \A\x_2,\: \nabla_{\x_2}g_2(\x_1,\x_2) = \A^\mathsf{T}\x_1\\
        &\|\nabla g_2(\x_1,\x_2)\|_2^2 \leq \|\A^\mathsf{T} \x_1\|_2^2 + \|\A \x_2\|_2^2 \leq 2 \|\A\|_{\text{op}}^2.
    \end{align*}
    Conditioned on $\A$, by symmetry $\y_1^\mathsf{T} \A \y_2$ and $-\y_1^\mathsf{T} \A \y_2$ give 
    \begin{align*}
        \E\left[g_2(\y_1,\y_2) | \A\right] = 0,
    \end{align*}
    and thus by Theorem~\ref{thm:lipschitz_sphere}, we find
    \begin{align*}
        \mathbb{P}\left\{ |\y_1^\mathsf{T} \A \y_2|\geq t \:| \A\right\} \leq 2 \exp \left( - \frac{cdt^2}{\|\A\|_{\mathrm{op}}^2}\right).
    \end{align*}
    Similar to step 1, this allows us to conclude that
    \begin{align*}
        \y_1^\mathsf{T} \A \y_2 \xrightarrow[d\to \infty]{\text{a.s.}} 0.
    \end{align*}
\end{proof}

\section{Computation of limits}
\label{appendixB}
\begin{lemma}[Herbst \cite{anderson2010introduction}]
\label{lem:herbst}
    Let $G$ be a Lipschitz function on $\mathbb{R}^m$, with Lipschitz constant $\|G\|_{\text{Lip}}$.
    Then, under the Gaussian measure $\gamma_m$, for all real scalars $\delta>0$, we get the following concentration:
    \begin{align*}
        \gamma_m\left( |G - \E[G]| \geq \delta\right) \leq 2 \exp\left(-\frac{\delta^2}{2\|G\|^2_{\text{Lip}}}\right).
    \end{align*}
\end{lemma}

\begin{lemma}
\label{lem:limit_type1}
Let $\X \in \mathbb{R}^{d\times n}$ have \ac{iid} Gaussian entries $\mathcal{N}(0,1)$. Partition the columns as $X = \begin{bmatrix}
    \X_1 & \X_2
\end{bmatrix}$, with $\X_1 \in \mathbb{R}^{d\times n_1}$, $\X_2 \in \mathbb{R}^{d\times (n-n_1)}$, and $n_1 \sim B(n,\pi_1)$ an independent random variable sampled from the Binomial distribution. Set
\begin{align*}
    \Sb_n &:= \frac{\X\X^\mathsf{T}}{n}, &&\R_n(z):= (z\I- \Sb_n)^{-1}, \text{and } &&\A_n(z) := \R_n(z)\frac{\X_1\X_1^\mathsf{T}}{n_1},
\end{align*}
where $z \in \mathbb{C}\setminus \{\lambda_1(\Sb_n), \ldots, \lambda_d(\Sb_n)\}$.
%, with $\text{Im}(z) > 0$.
Assume $d,n \to \infty$ with $d/n \to \gamma \in (0, \infty)$. Then, 
\begin{align*}
    \frac{1}{d}Tr(\A_n(z)) \xrightarrow[]{\text{a.s.}} T(z),
\end{align*}
where $T(z) = \int_t \frac{t}{z-t} d\mu_{\text{MP}(\gamma)}(t)$ is the $T$-transform of the \acl{MP} law of parameter $\gamma$. 
Moreover, this convergence is uniform on $\mathcal{K}(\eta) = \{z\in \mathbb{C},\: d(z,[a,b])> \eta\}$ for all real scalars $\eta>0$.
\end{lemma}
\begin{proof}
% Since the columns of $\X$ are identically distributed, the average contribution of the $n_1$ columns in $\X_1$ is equal to the average contribution of all the $n$ columns. 
We define for the $j$th column $\mathbf{x}_j$ of $\X$: 
\begin{align*}
    a_j := \frac{1}{d}\mathbf{x}_j^\mathsf{T} \R_n(z) \mathbf{x}_j.
\end{align*}
Then, 
\begin{align*}
    \frac{1}{d} \Tr \left(\R_n(z) \frac{\X_1\X_1^\mathsf{T}}{n_1}\right) &= \frac{1}{n_1} \sum_{j=1}^{n_1} a_j, \quad\quad\text{and} \quad\quad\frac{1}{d} \Tr \left(\R_n(z) \frac{\X\X^\mathsf{T}}{n}\right) = \frac{1}{n} \sum_{j=1}^n a_j.
\end{align*}
Since the columns of $\X$ are \ac{iid}, $\E[a_j]$ is the same for every $j$. Hence, 
\begin{align}
    \E \left[\frac{1}{d} \Tr \left(\R_n(z) \frac{\X_1\X_1^\mathsf{T}}{n_1}\right)\right] = \E \left[\frac{1}{d} \Tr \left(\R_n(z) \frac{\X\X^\mathsf{T}}{n}\right)\right].
    \label{eq:expectation_equality}
\end{align}
The RHS of \eqref{eq:expectation_equality} corresponds to $\E \int_t \frac{t}{z-t} d\mu_{\Sb_{d,n}}(t)$, where $\mu_{\Sb_{d,n}}(t)$ is the \ac{ESD} of $\Sb_{d,n}$. By the \acl{MP} theorem, the \ac{ESD} of $\Sb_{d,n}$ converges weakly (and in expectation) to the \acl{MP} law $\mu_{\text{MP}(\gamma)}$. Therefore, 
\begin{align}
    \E \left[\frac{1}{d} \Tr \left(\R_n(z) \frac{\X_1\X_1^\mathsf{T}}{n_1}\right)\right] \xrightarrow[d,n\to\infty , \frac{d}{n}\to \gamma]{} \int_t \frac{t}{z-t} d \mu_{\text{MP}(\gamma)}(t) = T(z).
    \label{eq:expectation-convergence}
\end{align}
We then show concentration using Lemma~\ref{lem:herbst} for Lipschitz functions.
Let $Y_n := \frac{1}{d} \Tr(\A_n(z)) = \frac{1}{d n_1} \sum_{j=1}^{n_1} \x_j^\mathsf{T} \R_n(z) \x_j$. Then,
\begin{align*}
    \frac{\partial \R_n(z)}{\partial \X_{kl}} &= \frac{1}{n} \R_n(z) \left(\X \mathbf{e}_l\mathbf{e}_k^\mathsf{T} + \mathbf{e}_k\mathbf{e}_l^\mathsf{T} \X^\mathsf{T}\right) \R_n(z) = \frac{1}{n} \R_n(z) \left(\x_l \mathbf{e}_k^\mathsf{T} + \mathbf{e}_k \x_l^\mathsf{T} \right) \R_n(z)\\
    \frac{\partial Y_n}{\partial \X_{kl}} &= \frac{1}{dn_1} \sum_{j=1}^{n_1} 2 \mathbf{e}_k^\mathsf{T} \R_n(z) \x_j \ind_{\{j=l\}} + \x_j^\mathsf{T} \frac{\partial \R_n(z)}{\partial \X_{kl}} \x_j\\
    &= \frac{2}{dn_1} \mathbf{e}_k^\mathsf{T} \R_n(z) \x_l \ind_{\{l\leq n_1\}} + \frac{1}{dn_1n} \sum_{j=1}^{n_1} \x_j^\mathsf{T} \R_n(z)(\x_l \mathbf{e}_k^\mathsf{T} + \mathbf{e}_k \x_l^\mathsf{T})\R_n(z) \x_j\\
    &= \frac{2}{dn_1} \mathbf{e}_k^\mathsf{T} \R_n(z) \x_l \ind_{\{l\leq n_1\}} + \frac{2}{dn_1n} \mathbf{e}_k^\mathsf{T} \R_n(z) \left( \sum_{j=1}^{n_1} \x_j\x_j^\mathsf{T}\right) \R_n(z) \x_l.
\end{align*}
We use the operator norm bound for the resolvent:
\begin{align*}
    \|\R(z)\|_{\mathrm{op}} &\leq \frac{1}{|\mathrm{Im}(z)|} := c(z).
\end{align*}
This way, we can bound the absolute value of the partial derivatives of $Y$:
\begin{align*}
    \left|\frac{\partial Y_n}{\partial \X_{kl}}  \right|^2 \leq  \frac{c_1}{d^2 n_1^2} \left( (\mathbf{e}_k^\mathsf{T} \R_n(z)\x_l)^2 \ind_{\{l\leq n_1\}}+ \frac{1}{n^2} (\mathbf{e}_k^\mathsf{T} \R_n(z)\X_1\X_1^\mathsf{T}\R_n(z)\x_l)^2\right),
\end{align*}
for some absolute constant $c_1 \geq 0$.
Summing the last equation over $k,l$ (and switching indices), we get: 
\begin{align}
    \|\nabla Y_n\|_2^2 
    &\leq \frac{c_1}{d^2 n_1^2}\left(\sum_{i=1}^{n_1} \|\R_n(z)\x_i\|_2^2+ \frac{1}{n^2} \sum_{j=1}^n\|\R_n(z)\X_1\X_1^\mathsf{T}\R_n(z)\x_j\|_2^2\right) \nonumber \\
    &\leq c_1\left( \frac{1}{d^2n_1^2} c(z)^2 \sum_{i=1}^{n_1}\|\x_i\|_2^2+ \frac{1}{d^2n_1^2 n^2} c^4(z)\|\X_1\|_{\mathrm{op}}^4 \sum_{j=1}^n \|\x_j\|_2^2\right).
    \label{eq:Y_n_gradient_norm}
\end{align}
We now define the sets $\mathcal{A}, \mathcal{B}, \mathcal{C}$, and $\mathcal{D}$, where the gradient is bounded:
\begin{align*}
    \mathcal{A} &:= \left\{\|\x_i\|_2 \leq 2 \sqrt{d},\: 1 \leq i \leq n \right\} ,\\
    \mathcal{B} &:= \left\{\|\X_1\|_{\text{op}} \leq c_3 \left(\sqrt{d} + \sqrt{n_1}\right)\right\} ,\\
    \mathcal{C} &:= \left\{ \left|\frac{n_1}{n} - \pi_1\right| \leq \varepsilon \right\}, & \varepsilon \in (0,1) ,\\
    \mathcal{D} &:= \mathcal{A} \cap \mathcal{B} \cap \mathcal{C} ,
\end{align*}
which are such that
\begin{align}
    &\mathbb{P}\left\{\mathcal{A}^c\right\} \leq n \exp( - c_2 d),
    &\mathbb{P}\left\{\mathcal{B}^c\right\} \leq 2 \exp( - d),
    &&\mathbb{P}\left\{\mathcal{C}^c\right\} \leq 2 \exp \left(-c_6n\varepsilon^2\right).
    \label{eq:concentration_estimates}
\end{align}
The probability bounds for $\mathcal{A}$ can be found in chapter~3 of \cite{vershynin2018high}, for $\mathcal{B}$ in chapter~4 of \cite{vershynin2018high}, and for $\mathcal{C}$ using the Chernoff bound.
We thus get that
\begin{align*}
    \|\nabla Y_n\|_2^2 &\leq \mathcal{O}\left(\frac{1}{dn}\right) =: L_{d,n}^2,
\end{align*}
with a probability of at least $1 - \mathcal{O}\left(\exp\left(-c_4d\right)\right)$ for $n$ and $d$ large enough such that $\frac{d}{n} \to \gamma$ ($c_4$ is a constant dependent on $z$ and $\varepsilon$).

Let $f: \mathbb{R}^{d \times n} \to \mathbb{C}$ be such that $Y_n = f(\X)$. We define $\tilde{f}$ to be the McShane's extension of $f$:
\begin{align*}
    \tilde{f}(\X) = \inf_{\Z \in \mathcal{D}} f(\Z)+ L_{d,n} \|\X-\Z\|_\text{F}.
\end{align*}
This extension guarantees that $\tilde{f}$ is $L_{d,n}$-Lipschitz and that
\begin{align}
    \tilde{f}(\X) &= f(\X) &\forall \X \in \mathcal{D}.
    &\label{eq:macshane_extension}
\end{align}
We will now show that this extension gives $\E \left[|f(\X) - \tilde{f}(\X)|\right] = o(1)$.
Using \eqref{eq:macshane_extension} and Cauchy-Schwarz, we get:
\begin{align}
    \left|\E \left[f(\X) - \tilde{f}(\X)\right] \right|
    &= \left| \E \left[(f(\X) - \tilde{f}(\X)) \ind_{\X \in \mathcal{D}^c}\right] \right| \nonumber\\
    &\leq \sqrt{ \E \left[\left(f(\X) - \tilde{f}(\X)\right)^2\right] \mathbb{P}\left[\mathcal{D}^c\right]} \nonumber \\
    &\leq \sqrt{2} \sqrt{ \E f^2(\X)+ \E \tilde{f}^2(\X)}\sqrt{\mathbb{P}[\mathcal{D}^c]},
    \label{eq:mcshane_absolute_diff}
\end{align}
and
\begin{align*}
    f^2(\X) &=\left(\frac{1}{d}\Tr\left(\R_n(z)\frac{\X_1\X_1^T} {n_1}\right)\right)^2 \leq \left\|\R_n(z) \frac{\X_1\X_1^T}{n_1}\right\|^2_{\text{op}} \leq c(z) \frac{\|\X_1\|_{\text{op}}^4}{n_1^2} .
\end{align*}
Using the Dominated convergence theorem and the fact that $n_1/n \to \pi_1$ almost surely, we further obtain that
\begin{align}
    \lim_{\substack{n\to\infty,\\d\to\infty,\\\frac{d}{n}\to\gamma}} \E[f^2(\X)] &= \E\left[ \lim_{d,n}\E\left[ f^2(\X) \left|n_1\right. \right]\right] \leq c_2(z) \E \left[ \lim_{d,n} \left(\sqrt{\frac{d}{n_1}} +1 \right)^2\right] =: c_7,
    \label{eq:mcshane_bound1}
\end{align}
$\tilde{f}(\mathbf{0}) = 0$ since $\X=0$ gives $\Sb_n=\mathbf{0}$, and $\A_n=\mathbf{0}$.
Since $\tilde{f}$ is Lipschitz we get: 
\begin{align}
     | \tilde{f}(\X) | &\leq |f(\mathbf{0})| + L_{d,n} \|\X\|_{\text{F}} = L_{d,n} \|\X\|_{\text{F}} \nonumber\\
     \E \tilde{f}^2(\X) & \leq L_{d,n}^2 \E \|\X\|^2_{\text{F}} = \mathcal{O}\left(1\right) .
     \label{eq:mcshane_bound2}
\end{align}
Plugging \eqref{eq:concentration_estimates}, \eqref{eq:mcshane_bound1}, and \eqref{eq:mcshane_bound2} into \eqref{eq:mcshane_absolute_diff}, we get:
\begin{align}
    \E\left[|f(\X) - \tilde{f}(\X)|\right] = o(1).
    \label{eq:absolute_mcshane_small_o}
\end{align}
We are now ready to provide a concentration bound on $f(\X)$ around its expectation:
\begin{align*}
    \mathbb{P}\left[|f(\X)- \E f(\X)| \geq t\right] 
    &= \mathbb{P}\left[|f(\X)- \E f(\X)| \geq t , \X \in \mathcal{D}\right] + \mathbb{P}\left[|f(\X)- \E f(\X)| \geq t , \X \in \mathcal{D}^c\right] \\
    &\leq \mathbb{P}\left[|\tilde{f}(\X)- \E f(\X)| \geq t , \X \in \mathcal{D}\right] + \mathbb{P}[\mathcal{D}^c] \\
    &\leq \mathbb{P}\left[|\tilde{f}(\X)- \E f(\X)| \geq t \right] + \mathbb{P}[\mathcal{D}^c] \\
    &\leq \mathbb{P}\left[|\tilde{f}(\X)- \E \tilde{f}(\X)| \geq t - |\E [\tilde{f}(\X) - f(\X)]| \right] + \mathbb{P}[\mathcal{D}^c].
\end{align*}
For $n$ and $d$ large enough, we know from \eqref{eq:absolute_mcshane_small_o} that $\E \left[|f(\X) - \tilde{f}(\X)|\right] \leq \frac{t}{2}$.
An application of Lemma~\ref{lem:herbst} gives (for $n$ and $d$ large enough):
\begin{align}
    \mathbb{P} \left\{|Y_n - \E Y_n| \geq t \right\} &\leq 2 \exp\left(\frac{-t^2 d n }{c_5}\right) + \mathcal{O}\left(\exp\left(-c_4 d\right)\right),
    \label{eq:concentration_Y_n}
\end{align}
for some large enough constant $c_5$. By the Borel-Cantelli lemma, we get that $|Y_n - \E[Y_n]| \xrightarrow[d,n\to \infty, \frac{d}{n}\to \gamma]{\mathrm{a.s.}} 0$, and therefore combining with \eqref{eq:expectation-convergence} yields
\begin{align*}
    Y_n \xrightarrow[d,n \to \infty , \frac{d}{n} \to \gamma]{\mathrm{a.s.}} T(z).
\end{align*}
Note that for $n,d$ large enough, the quantities are uniformly bounded and equicontinuous on the set $\mathcal{K}(\eta) = \{z\in \mathbb{C},\: d(z,[a,b])> \eta\}$ for all $\eta > 0$. An application of Arzela-Ascoli's theorem then gives the uniform convergence $\mathcal{K}(\eta)$.
\end{proof}

\begin{lemma}
\label{lem:limit_type2}
Let $\X \in \mathbb{R}^{d\times n}$ have \ac{iid} Gaussian entries $\sim \mathcal{N}(0,1)$. Partition the columns as $\X = \begin{bmatrix}
    \X_1 & \X_2 & \ldots & \X_K
\end{bmatrix}$, with $\X_i \in \mathbb{R}^{d\times n_i}$, $(n_1,\ldots, n_K) \sim \text{Mult}(n,\pi_1, \ldots, \pi_K)$, an independent random vector sampled from the multinomial distribution. For $i,j \in [K]$, set:
\begin{align*}
    \Sb_n &:= \frac{\X\X^\mathsf{T}}{n}, &&\R_n(z):= (z\I- \Sb_n)^{-1}, \text{and }&&\Bb_n(z) := \frac{\X_i\X_i^\mathsf{T}}{n_i} \R_n(z)\frac{\X_j\X_j^\mathsf{T}}{n_j},
\end{align*}
where $z \in \mathbb{C}\setminus \{\lambda_1(\Sb_n), \ldots, \lambda_d(\Sb_n)\}$.
%Where $z \in \mathbb{C}$, with $\text{Im}(z) > 0$.
Assume $d,n \to \infty$ with $d/n \to \gamma \in (0, \infty)$. Then, 
\begin{align*}
    \frac{1}{d}\Tr(\Bb_n(z)) \xrightarrow[]{\text{a.s.}} T(z) \left(\frac{\gamma}{\pi_i} \ind_{i=j} + \frac{1}{1-\gamma G(z)}\right),
\end{align*}
where $T(z) = \int_t \frac{t}{z-t} d\mu_{\text{MP}(\gamma)}(t)$, $G(z) = \int_t \frac{1}{z-t} \mu_{\text{MP}(\gamma)}(t)$ are respectively the $T$- and the Cauchy transform of the \acl{MP} law of parameter $\gamma$.
Moreover, this convergence is uniform on $\mathcal{K}(\eta) = \{z\in \mathbb{C},\: d(z,[a,b])> \eta\}$ for all $\eta>0$.
\end{lemma}
\begin{proof}
    We start with the case $i=j$. Without loss of generality, we take $i=1$.
    \begin{align}    
        \frac{1}{d}\E \Tr\left[\Bb_n(z)\right] 
        &= \frac{1}{d n_1^2} \E \Tr \left[\X_1\X_1^\mathsf{T} \left(zI - \frac{1}{n}\X\X^\mathsf{T}     \right)^{-1}\X_1\X_1^\mathsf{T}\right] \nonumber \\ 
        &= \frac{1}{dn_1^2} \E \sum_{i=1}^{d} \sum_{j,k=1}^{n_1}[\X_1]_{i,j} \left[\X_1^\mathsf{T} \R_n(z) \X_1\right]_{j,k} [\X_1]_{i,k} \nonumber \\
        &= \frac{1}{dn_1^2} \E \sum_{i=1}^{d} \sum_{j,k=1}^{n_1} \delta_{k=j} \left[\X_1^\mathsf{T} \R_n(z) \X_1\right]_{j,k} + [\X_1]_{i,k} \frac{\partial}{\partial  [\X_1]_{i,j}}\left[\X_1^\mathsf{T} \R_n(z)\X_1  \right]_{j,k}.
        \label{eq:differential_bi_side}
    \end{align}
    We used Stein's identity in the last step. To continue, we need the differential of $g(\X_1) = \X_1^\mathsf{T} \R(z) \X_1$, which is given by
    \begin{align}
        \mathcal{D}_g(\X_1)(\mathbf{H}) &=  \mathbf{H}^\mathsf{T} \R_n(z)\X_1 + \X_1^\mathsf{T} \R_n(z) \mathbf{H}+ \frac{1}{n}\X_1^\mathsf{T} \R_n(z) \left(\X_1 \mathbf{H}^\mathsf{T} + \mathbf{H}\X_1^\mathsf{T}\right) \R_n(z)\X_1.
    \end{align}
    Plugging this into \eqref{eq:differential_bi_side}, we get:
    \begin{align*}
        \frac{1}{d}\E \Tr\left[\Bb_n(z)\right] 
        &= \frac{1}{n_1}\E\Tr\left[\R_n(z)\frac{\X_1\X_1^\mathsf{T}}{n_1}\right] + \frac{1}{dn_1^2} \E\sum_{i=1}^{d} \sum_{j,k=1}^{n_1} [\X_1]_{i,k} [\R_n(z)\X_1]_{i,k} + [\X_1]_{i,k}[\X_1^\mathsf{T}\R_n(z)]_{j,i}\delta_{j=k} \\
        &+ 
    \frac{1}{n} \frac{1}{dn_1^2} \E\sum_{i=1}^{d} \sum_{j,k=1}^{n_1} \left( [\X_1]_{i,k}[\X_1^\mathsf{T}\R_n(z)\X_1]_{j,j} \: [\R_n(z)\X_1]_{i,k}+[\X_1]_{i,k}[\X_1^\mathsf{T}\R_n(z)]_{j,i}\: [\X_1^\mathsf{T}\R_n(z)\X_1]_{j,k}\right) \\
    &= \frac{1}{n_1}\E\Tr\left[\R_n(z)\frac{\X_1\X_1^\mathsf{T}}{n_1}\right] + \frac{1}{d} \E\Tr\left[\R_n(z)\frac{\X_1\X_1^\mathsf{T}}{n_1}\right] + \frac{1}{d n_1} \E\Tr\left[\R(z)\frac{\X_1\X_1^\mathsf{T}}{n_1}\right] \\
    &\quad+ \frac{1}{d n} \E\Tr\left[\R_n(z)\frac{\X_1\X_1^\mathsf{T}}{n_1}\right] \Tr\left[\R_n(z)\frac{\X_1\X_1^\mathsf{T}}{n_1}\right] + \frac{1}{d n} \E\Tr\left[\frac{\X_1\X_1^\mathsf{T}}{n_1} \R_n(z) \frac{\X_1\X_1^\mathsf{T}}{n_1} \R_n(z)\right].
    \end{align*}
    The third and last term vanishes in the limit $d \to \infty, n\to \infty$ since by Lemma~\ref{lem:limit_type1} the expectation of $\frac{1}{d} \E \Tr \left[\R_n(z) \frac{\X_1\X_1^\mathsf{T}}{n_1}\right]$ is finite in the limit.
    Using Hoffman-Wielandt inequality in the last term, we find:
    \begin{align}
        \Tr \left[ \frac{\X_1\X_1^\mathsf{T}}{n_1} \R_n(z) \frac{\X_1\X_1^\mathsf{T}}{n_1} \R_n(z)\right] &\leq \sum_{i=1}^d \lambda_i\left(\frac{\X_1\X_1^\mathsf{T}}{n_1}\right) \lambda_i\left(\R_n(z)\frac{\X_1 \X_1^\mathsf{T}}{n_1} \R_n(z)\right).
        &\label{eq:hoffman-wielandt}
    \end{align}
    Furthermore, we use the fact that for two positive semi-definite matrices $\mathbf{A}, \mathbf{B}$ we have: 
    \begin{align*}
        \lambda_i(\mathbf{A}\mathbf{B}) \leq \lambda_1(\mathbf{A}) \lambda_i(\mathbf{B}) \qquad\forall i.
    \end{align*}
    Using this in \eqref{eq:hoffman-wielandt}, we obtain:
    \begin{align}
        \Tr \left[ \frac{\X_1\X_1^\mathsf{T}}{n_1} \R_n(z) \frac{\X_1\X_1^\mathsf{T}}{n_1} \R_n(z)\right] &\leq \lambda_1^2\left(\R(z)\right) \sum_{i=1}^d \lambda_i^2\left(\frac{\X_1 \X_1^\mathsf{T}}{n_1}\right). \label{eq:hoffman-wielandt-2nd}
    \end{align}
    The convergence in expectations on the moments of the \acl{MP} distribution together with the bound $\lambda_1(\R(z)) \leq \frac{1}{|\text{Im}(z)|}$, we finally get from \eqref{eq:hoffman-wielandt-2nd}:
    \begin{align*}
        \frac{1}{d n} \E\Tr\left[\frac{\X_1\X_1^\mathsf{T}}{n_1} \R_n(z) \frac{\X_1\X_1^\mathsf{T}}{n_1} \R_n(z)\right] \xrightarrow[d,n\to \infty , \frac{d}{n}\to \gamma]{} 0.
    \end{align*}
    Lastly, we show that the variance of the normalized trace of 
    $Y_n := \frac{1}{d} \Tr\left( \R_n(z)\frac{\X_1\X_1^\mathsf{T}}{n_1}\right)$ 
    goes to 0. By the Gaussian Poincaré inequality:
    \begin{align}
        \Var\left(Y_n\right) &\leq \E\|\nabla Y_n\|^2. \label{eq:Gaussian-poincaré}
    \end{align}
    Using the bound \eqref{eq:Y_n_gradient_norm}, and the facts that $\E \|\x_i\|^2 = d$ and $\E \|\X_1\|_{\text{op}}^4 = \mathcal{O}(n^2)$ for $n$ and $d$ large enough, we obtain:
    \begin{align*}
        \E \| \nabla Y_n\|^2  = \mathcal{O}\left(\frac{1}{dn}\right).
    \end{align*}
    Together with \eqref{eq:Gaussian-poincaré}, this implies:
    \begin{align*}
        \Var\left(Y_n\right) \xrightarrow[d,n\to \infty, \frac{d}{n}\gamma]{} 0.
    \end{align*}
    % Lastly, we can use the concentration bounds \eqref{eq:concentration_Y_n} from the proof of lemma~\ref{lem:limit_type1} to prove that the variance of the nomalized trace of 
    % $Y_n := \frac{1}{d} \Tr\left( \R_n(z)\frac{\X_1\X_1^\mathsf{T}}{n_1}\right)$ 
    % goes to 0:
    % \begin{align*}
    %     \Var\left(Y_n\right) = \E \left[|Y_n - \E Y_n|^2 \right] = \int_0^\infty 2 t \mathbb{P} \left\{|Y_n - \E Y_n| \geq t \right\} d t \xrightarrow[d,n\to \infty, \frac{d}{n}\gamma]{} 0.
    % \end{align*}
    This means that:
    \begin{align*}
        \E \left(\frac{1}{d} \R_n(z) \frac{\X_1\X_1^\mathsf{T}}{n_1}\right)^2 \to \left(\E \frac{1}{d} \R_n(z) \frac{\X_1\X_1^\mathsf{T}}{n_1}\right)^2.
    \end{align*}
    All of this combined, we get: 
    \begin{align*}
        \frac{1}{d}\E \Tr\left[\Bb_n(z)\right] \xrightarrow[d,n\to \infty, \frac{d}{n}\to \gamma]{} T(z) \left(\frac{\gamma}{\pi_1} + 1 + \gamma T(z)\right).
    \end{align*}
Using the \acl{MP} relation with $G(z)$, the Cauchy transform yields 
    \begin{align}
        T(z) = z G(z) -1 = \frac{G(z)}{1-\gamma G(z)}, \label{eq:expectation-limit_type2}
    \end{align}
    and we find that 
    \begin{align*}
        T(z) \left(\frac{\gamma}{\pi_1} + 1 + \gamma T \right) = T(z) \left(\frac{\gamma}{\pi_1} + \frac{1}{1-\gamma G(z)}\right).
    \end{align*}
 This gives, with \eqref{eq:expectation-limit_type2}, the convergence in expectation.

We finish by using the same gradient-norm bound argument as in Lemma~\ref{lem:limit_type1} to show that $\|\nabla \frac{1}{d} \Tr(\Bb_n(z))\|^2=\mathcal{O}((dn)^{-1})$ with exponentially high probability. 
Herbst' lemma~\ref{lem:herbst}, together with the McShane extension and Borel–Cantelli lemma give almost surely convergence.
The case for $i\neq j$ can be deduced by using what was proven for $i=j$, using the following relation:
\begin{align*}
    &\frac{1}{d}\Tr \left[\frac{\X_i\X_i^\mathsf{T}}{n_i}\left(z\I - \Sb_n\right)^{-1}\frac{\X_j\X_j^\mathsf{T}}{n_j}\right] \\
    &= \frac{(n_i+n_j)^2}{2n_in_j} \frac{1}{d}\Tr \left[\frac{\X_i\X_i^\mathsf{T} + \X_j\X_j^\mathsf{T}}{n_i+n_j}(z\I - \Sb_n)^{-1} \frac{(\X_i\X_i^\mathsf{T}+\X_j\X_j^\mathsf{T})}{n_i+n_j}\right] \\
    &\quad - \frac{n_i^2}{2n_in_j} \frac{1}{d}\Tr \left[ \frac{\X_i\X_i^\mathsf{T}}{n_i}(z\I - \Sb_n)^{-1} \frac{\X_i\X_i^\mathsf{T}}{n_i}\right] - \frac{n_j^2}{2n_in_j} \frac{1}{d}\Tr \left[ \frac{\X_j\X_j^\mathsf{T}}{n_j}(z\I - \Sb_n)^{-1} \frac{\X_j\X_j^\mathsf{T}}{n_j}\right].
\end{align*}
The convergence of each term can be deduced from what we proved in the case $i=j$ and gives:
\begin{align*}
    &\frac{1}{d}\Tr \left[\frac{\X_i\X_i^\mathsf{T}}{n_i}\left(z\I - \Sb_n\right)^{-1}\frac{\X_j\X_j^\mathsf{T}}{n_j}\right] \\
    &\xrightarrow[d,n\to \infty, \frac{d}{n} \to \gamma]{\text{a.s.}} \frac{(\pi_i+\pi_j)^2}{2\pi_i \pi_j}T(z)\left(\frac{\gamma}{\pi_i + \pi_j} + \frac{1}{1-\gamma G(z)}\right) \\
    & \quad\quad\quad\quad\quad\quad-\frac{1}{2}\frac{\pi_i}{\pi_j}T(z)\left(\frac{\gamma}{\pi_i} + \frac{1}{1-\gamma G(z)}\right)  - \frac{1}{2} \frac{\pi_j}{\pi_i}T(z) \left(\frac{\gamma}{\pi_j} + \frac{1}{1-\gamma G(z)}\right)\\
    &\quad \quad\quad\quad\quad\quad= T(z) \left(\frac{1}{1-\gamma G(z)}\right).
\end{align*}
Similar to Lemma~\ref{lem:limit_type1}, for $n,d$ large enough, the quantities are uniformly bounded and equicontinuous on the set $\mathcal{K}(\eta) = \{z\in \mathbb{C},\: d(z,[a,b])> \eta\}$ for all $\eta > 0$. An application of Arzela-Ascoli's theorem thus gives the uniform convergence $\mathcal{K}(\eta)$.
\end{proof}

%%%%%%%%%%%%%%%%%%%%%%%%%%%%%%%%%%%%%%%%%%%%%%%%%%%%%%%%%%%%%%%%%%%%%%%%
%%%%%% Extension to infinite betas:
%%%%%%%%%%%%%%%%%%%%%%%%%%%%%%%%%%%%%%%%%%%%%%%%%%%%%%%%%%%%%%%%%%%%%%%%
\section{Rare and Intense Subpopulations}
\label{sec:rare_intense}

The main paper establishes Theorem~\ref{thm:SMM_phase_transition} for fixed signal strengths $\beta_k$ in the \ac{SMM}. In this appendix, we show that the same phase transition mechanism also extends to scenarios with diverging energy parameters, $\beta_k\to\infty$, provided that the probability-weighted energy $\pi_k\beta_k$ stays bounded. This signal model captures rare but intense subpopulations: a spike is seen in a vanishing fraction of observations ($\pi_k\to0$) yet it is strong when present ($\beta_k\to\infty$), with a fixed limiting weighted energy contribution $\pi_k\beta_k \to \eta_k$. In an application domain such as \ac{IMS}, this could represent a spatially scarce biological tissue structure with a strong mass spectral signature. The regime is of particular interest because it can potentially be encountered in application domains, yet it lies outside the framework provided by \cite{wang2024nonlinear}. Specifically, the concentration of quadratic forms that \cite{wang2024nonlinear} assumes fails in this scenario, while the finite-rank-based analysis provided in this paper still applies.

\subsection{The regime}
We take $\beta_k=\beta_k(d)$ and $\pi_k=\pi_k(d)$ such that, as $d,n\to\infty$ with
$d/n\to\gamma$,
\begin{equation}
\label{eq:rare_intense_regime}
\beta_k\to\infty,\qquad \pi_k\beta_k\to\eta_k\in(0,\infty),\qquad
\sqrt d\ \ll\ \beta_k\ \ll\ d .
\end{equation}
Since $\Lb_{l,m}=\theta_{l,m}\sqrt{\pi_l\pi_m\beta_l\beta_m}=\theta_{l,m}\sqrt{\eta_l\eta_m}$, the limiting matrix $\Lb$ depends on the spikes only through the weighted energy
$\eta_k$ and stays bounded.

\subsection{Failure of the concentration condition}
\begin{proposition}
\label{prop:5c_fails}
Under \eqref{eq:rare_intense_regime}, Assumption~5(c) of \cite{wang2024nonlinear} is violated. More concretely, using $\pi_1\beta_1 = \eta_1$, and $\A=\tfrac1d\I_d$, we have $\g^\mathsf{T}\A \g-\E[\g^\mathsf{T}\A \g]\not\prec\|\A\|_F$.
\end{proposition}

\begin{proof}
Since $\Tr\A=1$ and $\|\A\|_F=d^{-1/2}$, we have $\Tr\A/\|\A\|_F=\sqrt d$. Writing $\g^\mathsf{T}\A \g=\frac{1}{d}\|\g\|^2$, the unconditional mean is $\mu:= \E [\g^\mathsf{T}\A \g] =1+\tfrac1d\sum_{k=1}^K\pi_k\beta_k=1+O(d^{-1})$. \\
Conditioned on the rare spike $\{z=1\}$, $\g=\alpha\sqrt{\beta_1}\vb_1+\e$ with $\alpha\sim\mathcal N(0,1)$, $\e\sim\mathcal N(0,\I_d)$. Then,
\begin{align*}
    \g^\mathsf{T}\A \g\,\big|\,\{z=1\}
    =\frac{\alpha^2\beta_1}{d}+\frac{2\alpha\sqrt{\beta_1}\,\vb_1^\mathsf{T}\e}{d}+\frac{\|\e\|^2}{d}
    =1+\frac{\alpha^2\beta_1}{d}+O_\prec\!\Big(d^{-1/2}\Big),
\end{align*}
using $\frac{1}{d}\|\e\|^2=1+O_\prec(d^{-1/2})$ and
$\frac{1}{d}\,2\alpha\sqrt{\beta_1}\vb_1^T\e=O_\prec(\sqrt{\beta_1}/d)=o(d^{-1/2})$ since $\beta_1\ll d$. Since $\alpha^2\sim\chi^2_1$, we have $\mathbb{P}[\alpha^2\geq\frac{1}{2}]\geq c_0>0$. Moreover, on $\{z=1\}\cap\{\alpha^2\ge\tfrac12\}$ and for all large $d$,
\begin{align*}
    \big|\g^\mathsf{T}\A \g-\mu\big|\ \ge\ \frac{\beta_1}{2d}-o(d^{-1/2}),
    \qquad
    \frac{\big|\g^\mathsf{T}\A \g-\mu\big|}{\|\A\|_F}\ \gtrsim\ \frac{\beta_1/d}{d^{-1/2}}
    =\frac{\beta_1}{\sqrt d}\ \xrightarrow{\ \beta_1\gg\sqrt d\ }\ \infty .
\end{align*}
Hence, there exists $\varepsilon>0$ and large $d$ such that this event lies in $\{|\g^\mathsf{T}\A \g-\E[\g^\mathsf{T}\A \g]|>d^\varepsilon\|\A\|_F\}$, and its probability is
\begin{align*}
\mathbb{P}\left[z=1,\ \alpha^2\geq \frac{1}{2}\right]=\pi_1\mathbb{P}\left[\alpha^2\geq \frac{1}{2}\right]
\ \geq\ \frac{c_0\eta_1}{\beta_1}\ \asymp\ \frac{1}{\beta_1}.    
\end{align*}
Taking a real scalar $D=1$, since $\beta_1\ll d$ we have $1/\beta_1\gg d^{-1}=d^{-D}$, so the bound $\mathbb{P}[\cdot]<d^{-D}$ required by $\prec$ fails for this $D$. Thus, for this particular choice of $\A$, Assumption~5(c) of \cite{wang2024nonlinear} does not hold.
\end{proof}

%%%%%%%%%%%%%%%%%%%%%%%%%%%%%%%%%%%%%%%%%%%%
%%%%%%%%%%% Generalization argument
%%%%%%%%%%%%%%%%%%%%%%%%%%%%%%%%%%%%%%%%%%%%
\subsection{Recovery of the phase transition: a worked example}
\label{subsec:rare_intense_recovery} 

We explore a concrete $K=2$ instance in regime~\eqref{eq:rare_intense_regime}, which according to Proposition~\ref{prop:5c_fails} falls outside of the scope of \cite{wang2024nonlinear}, and for which the finite-rank-based argument provided in this paper still recovers the phase transition. The example is minimal: one rare-intense spike alongside an ordinary sub-critical one, with asymptotically orthogonal signals.\\

\paragraph{Setup}
For $K=2$, we draw $\vb_1,\vb_2$ from~\eqref{eq:vdefs} with $\theta_{1,2}=0$, so that $\vb_1,\vb_2$ are independent and uniform on $\mathbb S^{d-1}$. Conditioning on $\vb_1$, the map $\vb_2\mapsto\vb_1^\mathsf{T}\vb_2$ is $1$-Lipschitz with mean zero, so Theorem~\ref{thm:lipschitz_sphere} gives $\vb_1^\mathsf{T}\vb_2=O_\prec(d^{-1/2})$.

\begin{align}
\label{eq:rare_intense_example}
\beta_1=d^{1/2+\varepsilon}\ (\varepsilon\in(0,\tfrac12)),\quad
\pi_1=\frac{\eta_1}{\beta_1},\qquad
\beta_2>0\ \text{fixed},\quad \pi_2=1-\pi_1,\quad
\pi_2\beta_2\to \beta_2 := \eta_2<\sqrt\gamma .
\end{align}
Then, $\beta_1\to\infty$, $\pi_1\beta_1=\eta_1\in(0,\infty)$ fixed, and $\sqrt d\ll\beta_1\ll d$, so~\eqref{eq:rare_intense_regime} holds for the rare component. The second signal strength, $\beta_2$, adheres to the fixed-strength setting of
Theorem~\ref{thm:SMM_phase_transition}. The limiting Gram matrix is
diagonal, $\Lb=\mathrm{Diag}(\eta_1,\eta_2)$.\\

\paragraph{Block sizes}

With $(n_1,n_2)\sim\mathrm{Mult}(n;\pi_1,\pi_2)$, the rare count $n_1\sim\mathrm{Binomial}(n,\pi_1)$ has mean $m_1=n\pi_1\asymp(\eta_1/\gamma)d^{1/2-\varepsilon}\to\infty$. Since $\pi_1\to0$, the additive control $\{|n_1/n-\pi_1|\le\varepsilon\}$ of Lemmas~\ref{lem:limit_type1}--\ref{lem:limit_type2} is vacuous (its exponent $n\pi_1^2\asymp d^{-2\varepsilon}\to0$). We replace it by the multiplicative Chernoff bound, valid for every $(n,\pi_1)$ and $\delta\in(0,1)$,

\begin{align}
\label{eq:chernoff}
\mathbb{P}\left(|n_1-m_1|\geq \delta\,m_1\right)\leq 2\exp\left(-\tfrac{\delta^2 m_1}{3}\right).
\end{align}
As $m_1$ is a diverging power of $d$, the right-hand side is summable, and the Borel-Cantelli lemma (letting $\delta$ to $0$ along a countable sequence) yields
\begin{align}
\label{eq:rare_probability}
n_1\xrightarrow{ \text{a.s.}}\infty,\qquad \frac{n_1}{n\pi_1}\xrightarrow{\ a.s.\ }1, \qquad \frac{n_1}{n}\beta_1\xrightarrow{\ a.s.\ }1, \quad \frac{n_2}{n}\xrightarrow{\ a.s.\ }1 .    
\end{align}\\

\paragraph{Weighted trace (adapting the convergence results)}

The diagonal limit of Lemma~\ref{lem:limit_type2} carries a factor
$\gamma/\pi_i$ that diverges as $\pi_1\to0$, so the lemma cannot be applied directly to the rare block. It is, however, needed only in the $\tfrac{n_i}{n}$-weighted form, in which the weight cancels the divergence.
Setting
\begin{align*}
\Bb_n(z):=\frac{\X_1\X_1^\mathsf{T}}{n_1}\R_n(z)\frac{\X_1\X_1^\mathsf{T}}{n_1},
\qquad
\mathbf C_n(z):=\frac{\X_2\X_2^\mathsf{T}}{n_2}\R_n(z)\frac{\X_2\X_2^\mathsf{T}}{n_2},    
\end{align*}
the proof of Lemma~\ref{lem:limit_type2}---which uses $\pi_i$ only through $d/n_i\to\gamma/\pi_i$, all variance bounds needing merely $n_i\to\infty$ and $\|\R_n\|\le 1/|\mathrm{Im}(z)|$---together with~\eqref{eq:rare_probability} gives:
\begin{align}
    \frac{n_1}{n}\frac{1}{d}\Tr(\Bb_n(z)) &\xrightarrow[]{\text{a.s.}} T(z) \gamma, \label{eq:weighted_trace_1} \\
    \frac{n_2}{n}\frac{1}{d}\Tr(\mathbf{C}_n(z)) &\xrightarrow[]{\text{a.s.}} T(z) \left(\gamma +\frac{1}{1-\gamma G(z)}\right) , \label{eq:weighted_trace_2}
\end{align}
the rare weight $\tfrac{n_1}{n}\asymp\pi_1$ absorbing the $\gamma/\pi_1$.\\

\paragraph{Transform first, then take the limit}

As in the approach developed above, we let:
\begin{align*}
 \M_{d,n}(z) := \Pb_2^\mathsf{T} (z\I-\Z)^{-1}\Pb_1, \\
 \tilde{\M}_{d,n}(z) := \tilde{\Pb}_2^\mathsf{T} (z\I-\Z)^{-1}\tilde{\Pb}_1   .
\end{align*}
where $\tilde{\Pb}_i$ is obtained by scaling the even columns by $-1$.
The eigenvalues popping out the \ac{MP} spectra are the roots of
\begin{align}
\label{eq:sym_det}
|\det(\I_{2K}-\M_{d,n})|
=\left|\det\left(\I_{2K}-(\M_{d,n}+\tilde\M_{d,n}-\M_{d,n}\tilde\M_{d,n})\right)\right|^{1/2}.
\end{align}
In the approach developed in section~\ref{sec:cov-matrices+proof} for fixed $\beta_k    $, we first found the limiting $\M_{d,n}$, and manipulated the matrix afterward. Here, the entries of $\M_{d,n}$ diverge ($c_1\sim\sqrt{\beta_1} \to\infty$), so we instead perform the manipulation at finite $d$ and limit afterwards. Conjugating by
\begin{align*}
    \R:=\mathrm{Diag}\left(\sqrt{\frac{n_2}{n}},\sqrt{\frac{n_1}{n}},
\sqrt{\frac{n_1}{n}},\sqrt{\frac{n_2}{n}}\right) \in \mathbb{R}^{2K\times 2K}    
\end{align*}
is a similarity, hence determinant-preserving for every finite $d$, and it places every divergent quantity behind a weight $\tfrac{n_k}{n}\asymp\pi_k$. With $\Hb_k=c_k\Z_k+\tfrac{c_k^2}{2}(\vb_k^\mathsf{T}\Z_k\vb_k)\I$, the strength $\beta_1$ (equivalently $c_1\sim\sqrt{\beta_1}$) therefore enters every entry of $\R\left(\M_{d,n}+\tilde\M_{d,n}-\M_{d,n}\tilde\M_{d,n}\right)\R^{-1}$ only through the bounded combinations
\begin{align*}
    \tfrac{n_1}{n}c_1^2\to\eta_1,\qquad \frac{n_1}{n}c_1\to0,
\end{align*}
together with the weighted traces~\eqref{eq:weighted_trace_1}, \eqref{eq:weighted_trace_2}. Every entry is thus converging in the limit. The diagonal blocks converge to multiples of the identity, while the cross blocks carry the additional factor $\vb_1^\mathsf{T}\vb_2=O_\prec(d^{-1/2}) \to 0$ and hence vanish.
We then have
\begin{align*}
    &\R(\M+\tilde{\M})\R^{-1} \\
    &:= 
    \begin{bmatrix}
        2\frac{n_1}{n} \vb_1^\mathsf{T} (z\I-\Z)^{-1}\Hb_1\vb_1 & 0 & 2\frac{n_2}{n} \vb_1^\mathsf{T} (z \I-\Z)^{-1}\Hb_2 \vb_2  & 0 \\
        0 &  2\frac{n_1}{n}\vb_1^\mathsf{T}\Hb_1(z\I-\Z)^{-1}\vb_1 & 0 & 2\frac{n_1}{n}\vb_1^\mathsf{T}\Hb_1(z\I-\Z)^{-1}\vb_2  \\
        2\frac{n_1}{n} \vb_2^\mathsf{T} (z\I-\Z)^{-1}\Hb_1\vb_1 & 0 & 2\frac{n_2}{n} \vb_2^\mathsf{T} (z\I-\Z)^{-1}\Hb_2 \vb_2 & 0 \\
        0 &  2\frac{n_2}{n}\vb_2^\mathsf{T}\Hb_2(z\I-\Z)^{-1}\vb_1  & 0 & 2\frac{n_2}{n}\vb_2^\mathsf{T}\Hb_2(z\I-\Z)^{-1}\vb_2 \\
    \end{bmatrix}\\
    &\to \begin{bmatrix}
        \eta_1 G(z) & 0 & 0 & 0 \\
        0 & \eta_1 G(z) & 0 & 0 \\ 
        0 & 0 & 2 c_2 T(z) + c_2^2 G(z) & 0 \\
        0 & 0 & 0 & 2 c_2 T(z) + c_2^2 G(z)
    \end{bmatrix}\\
    &\R\M\tilde{\M}\R^{-1} \\
    &\to \begin{bmatrix}
        \eta_1 \gamma T(z)G(z) & 0 & 0 & 0 \\
        0 & \eta_1 \gamma T(z) G(z) & 0 &  \\
        0 & 0 & -c_2^2 T(z)\gamma G(z) & 0 \\
        0 & 0 & 0 & -c_2^2 T(z) \gamma G(z)
    \end{bmatrix}
\end{align*}
\begin{align}
    \R(\M+\tilde{\M})\R^{-1} -\R\M\tilde{\M}\R^{-1} &\to \begin{bmatrix}
        \eta_1 T(z) & 0 & 0 & 0 \\
        0 & \eta_1 T(z) & 0 & 0 \\
        0 & 0 & \beta_2 T(z) & 0 \\
        0 & 0 & 0 & \beta_2 T(z)
    \end{bmatrix},
    \label{eq:W_limit}
\end{align}
where we used the relation between $G(z)$ and $T(z)$ : $T(z) = G(z)/(1-\gamma G(z))$.\\
\paragraph{Phase transition}
By~\eqref{eq:sym_det} and~\eqref{eq:W_limit}, the limiting outlier equation is
\begin{align*}
    |1-\eta_1 T(z)|\,|1-\eta_2 T(z)|=0 .
\end{align*}
By the threshold analysis of Section~\ref{sec:schur-complement-equivalence}, $T(z)=1/\eta_k$ has a root outside $[a,b]$ iff $\eta_k>\sqrt\gamma$. Since $\eta_2<\sqrt\gamma$ by construction, the second factor gives no outlier and $\lambda_2(\Sb_{d,n})\to(1+\sqrt\gamma)^2$, while
\begin{align*}    
\lambda_1(\Sb_{d,n})\xrightarrow{\text{a.s.}}
\begin{cases}
T^{-1}(1/\eta_1)=(1+\eta_1)(1+\gamma/\eta_1), & \eta_1>\sqrt\gamma,\\
(1+\sqrt\gamma)^2, & \eta_1\le\sqrt\gamma.
\end{cases}
\end{align*}
Detectability of the rare spike's subpopulation is thus governed by its weighted energy $\eta_1=\pi_1\beta_1$ alone, independently of the peak strength $\beta_1=d^{1/2+\varepsilon}$, for every $\varepsilon\in(0,\tfrac12)$. Since this is a configuration on which Assumption~5(c) of \cite{wang2024nonlinear} fails by Proposition~\ref{prop:5c_fails}, it demonstrates merit for a finite-rank-based approach as pursued in this paper.

} %%%%%% End Appendix

\newpage
\bibliographystyle{IEEEtran}
\bibliography{reference}

\end{document}